\documentclass{birkjour}
\usepackage{psfrag, graphicx, array, cite, appendix, comment}
 \newtheorem{thm}{Theorem}[section]
 \newtheorem{conj}[thm]{Conjecture}

 \newtheorem{prop}[thm]{Proposition}
 \theoremstyle{definition}
 \newtheorem{defn}[thm]{Definition}
 \theoremstyle{remark}

 \numberwithin{equation}{section}
\def\TSC{{\footnotesize TSS-CMC} }
\def\SC{{\footnotesize SS-CMC} }
\def\sC{{\footnotesize S-CMC} }
\def\C{{\footnotesize CMC} }
\begin{document}

%
%
%
%
%
%

\title[CMC foliation in the Schwarzschild spacetime]{Spacelike spherically symmetric CMC foliation in the extended Schwarzschild spacetime}

\author[K.-W. Lee]{Kuo-Wei Lee}
\address{Department of Mathematics, National Taiwan University, Taipei, Taiwan}
\email{d93221007@gmail.com}


\author[Y.-I. Lee]{Yng-Ing Lee}
\address{Department of Mathematics, National Taiwan University, Taipei, Taiwan\\
National Center for Theoretical Sciences, Taipei Office, National Taiwan University, Taipei, Taiwan}
\email{yilee@math.ntu.edu.tw}



\date{\Today}

\begin{abstract}
We first summarize the characterization of smooth spacelike spherically symmetric constant mean curvature (\SC$\!\!$)
hypersurfaces in the Schwarzschild spacetime and Kruskal extension.
Then use the characterization to prove special \SC foliation property, and verify part of the conjecture by Malec and \'{O}~Murchadha in \cite{MO}.
\end{abstract}

\maketitle

\section{Introduction}
Spacelike constant mean curvature (\sC$\!\!$) hypersurfaces in spacetimes are important and interesting objects in general relativity.
From the geometric point of view,
\sC hypersurfaces in spacetimes are critical points of the surface area functional with fixed enclosed volume \cite{BCI}.
Local maximal properties of \sC hypersurfaces in some conditions are proved by Brill and Flaherty \cite{BF}.
These characterizations are similar to those of compact \C hypersurfaces in Euclidean spaces.

Brill, Cavallo, and Isenberg considered spacelike spherically symmetric constant mean curvature (\SC$\!\!$) hypersurfaces in static spacetimes,
especially in Schwarzschild spacetimes \cite{BCI}.
From the variational principle, the \C equation is derived.
They gave descriptions of the behavior of \SC hypersurfaces in the Schwarzschild spacetime through numerical integration and effective potential.

Among issues of \C hypersurfaces, \C foliations are important in understanding spacetimes and relativistic cosmology because York
\cite{Y} suggested the concept of the \C time functions on spacetimes.
Marsden and Tipler considered the existence and uniqueness of \C Cauchy hypersurface foliations
with mean curvature as a parameter in spatially closed universes or asymptotically flat spacetimes \cite{MT}.
In the paper \cite{BCI},
Brill-Cavallo-Isenberg conjectured that a complete \C foliation in the extended Schwarzschild spacetime with mean curvature varied for all values can be obtained.
This conjecture is answered by Eardly and Smarr on the existence in \cite{ES}, and Pervez, Qadir,
and Siddiqui gave a procedure which possibly produce an explicit construction with numerical evidence \cite{PQS}.

Malec and \'{O} Murchadha also considered \SC hypersurfaces and \C foliations in the Schwarzschild spacetime \cite{MO, MO2}.
Their idea is viewing the Einstein equation as a dynamical system,
then the Hamiltonian and momentum constraints give the formula of the second fundamental form of \SC hypersurfaces.
Through the analysis of the lapse function and the mean curvature of spherical two-surfaces,
they can characterize the behavior of \SC hypersurfaces.
In addition, they suggested two types of \SC foliation.
In \cite{MO}, they conjectured that the extended Schwarzschild spacetime can be foliated
by a family of \SC hypersurfaces with fixed mean curvature but varied another parameter.
In \cite{MO2}, they described another \SC foliation with varied mean curvature and the parameter.
Both \C foliations have phenomenon of exponentially collapsing lapse.

In this paper, we investigate  the \SC foliation property with fixed mean curvature and partially answer the conjecture posted by Malec and \'{O} Murchadha in
\cite{MO}.
To achieve the goal, detail study on properties of \SC hypersurfaces in the Schwarzschild spacetime and Kruskal extension is necessary.
Before aware of the work of Malec and \'{O} Murchadha in \cite{MO},
we characterize all smooth \SC hypersurfaces in the Kruskal extention from different points of view in \cite{LL1}.
Our proof of the \SC foliation property highly depends on the explicit formulation obtained in \cite{LL1}.
For the reader's reference, we first summarize related results on the smooth
\SC hypersurfaces in the Schwarzschild spacetime and Kruskal extension in section~\ref{Kru}.

In section~\ref{CMCfoliation}, we concentrate on the foliation of \SC family with
$T$-axisymmetry in the Kruskal extension, which is abbreviated by \TSC for convenience.
We explain and reformulate the \TSC foliation conjecture in section~\ref{conjecture},
and then derive criteria for \TSC family being disjoint in section~\ref{criteria}.
These criteria are used to show that the \TSC foliation holds in the Kruskal extension region
{\tt I$\!$I} and {\tt I$\!$I'} as in Theorem~\ref{CMCfoliation1} and Theorem~\ref{localCMCfoliation34}.
For the foliations in region {\tt I} and {\tt I'} with nonzero mean curvature, the estimates of the criteria are more subtle.
We test some crucial cases by the numerical integrations and it looks that \TSC foliation might hold in general.
As the analysis for this part is technically more involving, we leave the investigation to the future.

When mean curvature is zero, the surface is called maximal hypersurface.
In this case, it is much easier to prove the foliation property in region {\tt I} and {\tt I'}
and we show that {\small TSS}-maximal hypersurface form a foliation in the whole Kruskal extension.
This result was first proved in \cite{BO} by different arguments.

\section{Preliminary}\label{Kru}
The Schwarzschild spacetime is a $4$-dimensional time-oriented Lorentzian manifold with metric
\begin{align*}
\mathrm{d}s^2=-\left(1-\frac{2M}r\right)\mathrm{d}t^2+\frac1{\left(1-\frac{2M}r\right)}\,\mathrm{d}r^2
+r^2\,\mathrm{d}\theta^2+r^2\sin^2\theta\,\mathrm{d}\phi^2.
\end{align*}
After coordinates change, the Schwarzschild metric can be written as
\begin{align}
\mathrm{d}s^2&=\frac{16M^2\mathrm{e}^{-\frac{r}{2M}}}{r}(-\mathrm{d}T^2+\mathrm{d}X^2)+r^2\,\mathrm{d}\theta^2+r^2\sin^2\theta\,\mathrm{d}\phi^2 \notag \\
&=\frac{16M^2\mathrm{e}^{-\frac{r}{2M}}}{r}\,\mathrm{d}U\mathrm{d}V+r^2\,\mathrm{d}\theta^2+r^2\sin^2\theta\,\mathrm{d}\phi^2, \label{SchMetric2}
\end{align}
where
\begin{align}
\left\{
\begin{array}{l}
\displaystyle(r-2M)\,\mathrm{e}^{\frac{r}{2M}}=X^2-T^2=VU\\
\displaystyle\frac{t}{2M}=\ln\left|\frac{X+T}{X-T}\right|=\ln\left|\frac{V}{U}\right|. \label{trans}
\end{array}
\right.
\end{align}
It shows that $r=2M$ is only a coordinate singularity.
The Schwarzschild spacetime has a maximal analytic extension, called the Kruskal extension.
It is the union of regions {\tt I}, {\tt I$\!$I}, {\tt I'}, and {\tt I$\!$I'},
where regions {\tt I} and {\tt I$\!$I} correspond to the exterior and interior of one Schwarzschild spacetime, respectively,
and regions {\tt I'} and {\tt I$\!$I'} correspond to the exterior and interior of the other Schwarzschild spacetime.
Figure~\ref{KruskalSimple} points out their correspondences and coordinate systems $(X,T)$ or $(U,V)$.

\begin{figure}[h]
\psfrag{A}{\tt I}
\psfrag{B}{\tt I$\!$I}
\psfrag{C}{\tt I$\!$I'}
\psfrag{D}{\tt I'}
\psfrag{E}{$U$}
\psfrag{F}{$V$}
\psfrag{X}{$X$}
\psfrag{T}{$T$}
\psfrag{r}{$r$}
\psfrag{t}{$t$}
\psfrag{M}{\tiny$2M$}
\psfrag{P}{\tiny $\partial_T$}
\centering
\includegraphics[height=50mm,width=71mm]{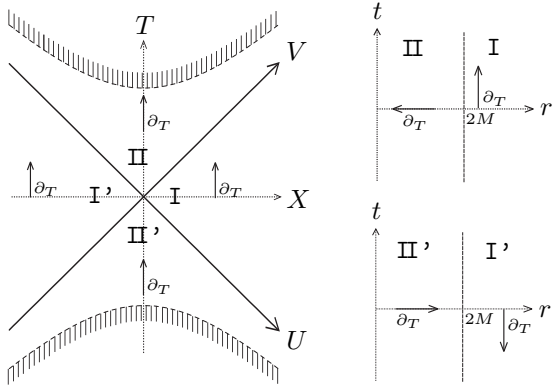}
\caption{The Kruskal extension of Schwarzschild spacetimes.}\label{KruskalSimple}
\end{figure}
When $r=2M$, relation (\ref{trans}) implies $U=0$ or $V=0$.
Furthermore, the solution of (\ref{trans}) with $r=2M$ and any finite value $t$ is $(U,V)=(0,0)$,
so the origin of the Kruskal extension correspond to all points of $r=2M$ and $t$ finite value
(although these points are not defined in the Schwarzschild spacetime).
Similarly, if $t\to\infty$ (or $-\infty$), then (\ref{trans}) gives $U=0$ (or $V=0$).
Hence $V$-axis correspond to all points $r=2M$ and $t=\infty$,
and $U$-axis correspond to all points $r=2M$ and $t=-\infty$

Sometimes we will use null coordinates $(u,v)$ by
\begin{align}
u=t-(r+2M\ln|r-2M|)\quad\mbox{and}\quad v=t+(r+2M\ln|r-2M|), \label{NullCoordinatesuv}
\end{align}
and relations between $(U,V)$ and $(u,v)$ are given by
\begin{align*}
\begin{array}{ccccc}
& \quad \mbox{region {\tt I}} & \quad\mbox{region {\tt I$\!$I}} & \quad \mbox{region {\tt I'}}
& \quad \mbox{region {\tt I$\!$I'}}\\
\hline
U & \quad \mathrm{e}^{-\frac{u}{4M}} & \quad -\mathrm{e}^{-\frac{u}{4M}}
& \quad -\mathrm{e}^{-\frac{u}{4M}} & \quad \mathrm{e}^{-\frac{u}{4M}} \\
V &\quad \mathrm{e}^{\frac{v}{4M}}  & \quad \mathrm{e}^{\frac{v}{4M}}
&\quad -\mathrm{e}^{\frac{v}{4M}}  & \quad-\mathrm{e}^{\frac{v}{4M}}. \\
\end{array}
\end{align*}

In this article, we will take $\partial_T$ as a future directed timelike vector field.
Note that $\partial_T$ in the two Schwarzschild spacetimes has different directions and it is indicated in Figure~\ref{KruskalSimple}.

In the following subsections,
we will summarize the results of spacelike spherically symmetric constant mean curvature (\SC for short) hypersurfaces in
Schwarzschild spacetimes and Kruskal extension.
These formulae and arguments are useful when dealing with \C foliation problem.
We refer to our article in ArXiv \cite{LL1} for more details.

\subsection{SS-CMC solutions in region {\tt I} and {\tt I'}}
To understand \SC hypersurfaces in the Kruskal extension,
one can study and analyze \SC solutions in the Schwarzschild spacetime first,
and then discuss their images in the Kruskal extension.

First, we consider \SC hypersurfaces in the Schwarzschild exterior, which map to the region {\tt I} or {\tt I'} in the Kruskal extension.
Since a spacelike hypersurface in the Schwarzschild exterior can always be written as a graph of $r,\theta$, and $\phi$,
particularly, an \SC hypersurface in the Schwarzschild exterior is a graph of $f(r)$.
\footnote{The function $f(r)$ corresponds to the height function $h(r)$ in paper~\cite{MO}.}
We use subscripts $f_1$ and $f_3$ to represent \SC hypersurfaces that map to region {\tt I} and {\tt I'}, and leave subscripts $f_2$ and $f_4$
for \SC hypersurfaces that map to region {\tt I$\!$I} and {\tt I$\!$I'}.

For $f_1(r)$, the constant mean curvature equation is
\begin{align*}
f_{1}''+\left(\left(\frac1{h}-(f_{1}')^2h\right)\left(\frac{2h}{r}+\frac{h'}2\right)+\frac{h'}{h}\right)f_{1}'-3H\left(\frac1{h}-(f_{1}')^2h\right)^{\frac32}=0,
\end{align*}
where $h(r)=1-\frac{2M}{r}$ and $H$ is the constant mean curvature.~\footnote{In paper~\cite{MO},
they use the terminology extrinsic curvature $K$, and the relation is $H=\frac{K}{3}$.}
This is a second order ordinary differential equation, and we can solve the equation as follows:
\begin{prop}{\rm\cite{LL1}} \label{scmc3prop1}
Suppose $\Sigma^{1}=(f_{1}(r),r,\theta,\phi)$ is an \SC hypersurface in the Schwarzschild exterior
{\rm(}corresponding to region {\tt I}{\rm)} with constant mean curvature $H$.
Then
\begin{align*}
f_{1}(r;H,c_{1},\bar{c}_1)=\int_{r_1}^r\frac{1}{h(x)}\frac{l_1(x;H,c_{1})}{\sqrt{1+l_1^2(x;H,c_{1})}}\,\mathrm{d}x+\bar{c}_1,
\end{align*}
where
\begin{align*}
l_1(r;H,c_{1})=\frac{1}{\sqrt{h(r)}}\left(Hr+\frac{c_1}{r^2}\right).
\end{align*}
Here $c_1$ and $\bar{c}_1$ are constants, and $r_1\in(2M,\infty)$ is fixed.
\end{prop}
Similarly, the constant mean curvature equation of an \SC hypersurface $\Sigma^{3}=(f_{3}(r),r,\theta,\phi)$ (corresponding to region {\tt I'}) is
\begin{align}
f_{3}''+\left(\left(\frac1{h}-(f_{3}')^2h\right)\left(\frac{2h}{r}+\frac{h'}2\right)+\frac{h'}{h}\right)f_{3}'+3H\left(\frac1{h}-(f_{3}')^2h\right)^{\frac32}=0,
\label{cmceq4}
\end{align}
and the solution of equation (\ref{cmceq4}) is the following:
\begin{prop}{\rm\cite{LL1}}
Suppose $\Sigma^{3}=(f_{3}(r),r,\theta,\phi)$ is an \SC hypersurface in the Schwarzschild exterior
{\rm(}corresponding to region {\tt I'}{\rm)} with constant mean curvature $H$.
Then
\begin{align*}
f_{3}(r;H,c_{3},\bar{c}_3)
=\int_{r_3}^r\frac{1}{h(x)}\frac{l_{3}(x;H,c_{3})}{\sqrt{1+l_{3}^2(x;H,c_{3})}}\,\mathrm{d}x+\bar{c}_3,
\end{align*}
where
\begin{align*}
l_{3}(r;H,c_3)=\frac{1}{\sqrt{h(r)}}\left(-Hr-\frac{c_{3}}{r^2}\right).
\end{align*}
Here $c_3$ and $\bar{c}_3$ are constants, and $r_3\in(2M,\infty)$ is fixed.
We remark that for given constant mean curvature $H$, if $c_1=c_3$, then $f_1'(r)=-f_3'(r)$.
\end{prop}

We have complete discussion about the asymptotic behavior of \SC hypersurfaces $\Sigma_1$ and $\Sigma_3$ in \cite{LL1}.
Here we just remark that for $H\neq 0$, \SC hypersurfaces are asymptotic null and for $H=0$,
maximal hypersurfaces are asymptotic to the spatial infinity.

\subsection{SS-CMC solutions in region {\tt I$\!$I} and {\tt I$\!$I'}} \label{32}
For \SC hypersurfaces in the Schwarzschild interior,
since the future timelike direction is $-\partial_r$ direction,
spacelike condition gives that an \SC hypersurface can be written as $(t,g(t),\theta,\phi)$ for some function $r=g(t)$.

\begin{prop}{\rm\cite{MO}}\label{cylindrical}
Each constant slice $r=r_0 \in(0,2M)$ in the Schwarzschild interior {\rm(}corresponding to region {\tt I$\!$I}{\rm)} is an \SC hypersurface with mean curvature
\begin{align*}
H(r_0)=\frac{2r_0-3M}{3\sqrt{r_0^3(2M-r_0)}}.
\end{align*}
These hypersurfaces are called cylindrical hypersurfaces.
\end{prop}

For $r=g(t)\neq\mbox{constant}$, we consider its inverse function, and denote $t=f_{2}(r)$ whenever it is defined.
Since $f_2(r)$ is obtained from the inverse function, we have $f_2'(r)\neq 0$ and will allow $f_2'(r)=\infty$ or $-\infty$.
\begin{prop}{\rm\cite{LL1}}\label{prop6}
Suppose $\Sigma^{2}=(f_{2}(r),r,\theta,\phi)$ is an \SC hypersurface in the Schwarzschild interior {\rm(}corresponding to region {\tt I$\!$I}{\rm)}.
Then
\begin{align}
f_{2}^*(r;H,c_{2},\bar{c}_2)
&=\int_{r_2}^r\frac{1}{-h(x)}\sqrt{\frac{l^2_{2}(x;H,c_{2})}{l_{2}^2(x;H,c_{2})-1}}\,\mathrm{d}x+\bar{c}_2,
\quad\mbox{or}
\label{f2positive} \\
f_{2}^{**}(r;H,c_{2},\bar{c}_2')
&=\int_{r_2'}^r\frac{1}{h(x)}\sqrt{\frac{l^2_{2}(x;H,c_{2})}{l_{2}^2(x;H,c_{2})-1}}\,\mathrm{d}x+\bar{c}_2'
\label{f2negative}
\end{align}
depending on the sign of $f_2'(r)$, where
\begin{align*}
l_{2}(r;H,c_2)=\frac1{\sqrt{-h(r)}}\left(-Hr-\frac{c_{2}}{r^2}\right).
\end{align*}
Here $c_2, \bar{c}_2, \bar{c}_2'$ are constants, and $r_2,r_2'$ are points in the domain of $f_2^*(r)$ and $f_2^{**}(r)$, respectively.

The function $l_2(r)$ should satisfy $l_2(r)>1$, which implies $c_{2}<0$ when $H>0$ and $c_{2}<-8M^3H$ when $H\leq 0$.
We will write $f_2(r)$ to denote both $f_2^{*}(r)$ and $f_2^{**}(r)$.
\end{prop}

Similarly, for \SC hypersurfaces in another Schwarzschild interior (corresponding to region {\tt I$\!$I'}),
all cylindrical hypersurfaces $r=r_0\in(0,2M)$ are \SC solutions with mean curvature $H(r_0)=\frac{3M-2r_0}{3\sqrt{r_0^3(2M-r_0)}}$.
When $r\neq\mbox{constant}$,
we have the following results.
\begin{prop}{\rm\cite{LL1}}\label{propinteriorII}
Suppose $\Sigma^{4}=(f_{4}(r),r,\theta,\phi)$ is an \SC hypersurface in the Schwarzschild interior {\rm(}corresponding to region {\tt I$\!$I'}{\rm)}.
Then
\begin{align}
f_{4}^{*}(r;H,c_{4},\bar{c}_4)
&=\int_{r_4}^r\frac{1}{-h(x)}\sqrt{\frac{l^2_{4}(x;H,c_{4})}{l_{4}^2(x;H,c_{4})-1}}\,\mathrm{d}x+\bar{c}_4,
\quad\mbox{or} \label{f4positive} \\
f_{4}^{**}(r;H,c_{4},\bar{c}_4')
&=\int_{r_4'}^r\frac{1}{h(x)}\sqrt{\frac{l^2_{4}(x;H,c_{4})}{l_{4}^2(x;H,c_{4})-1}}\,\mathrm{d}x+\bar{c}_4'
\label{f4negative}
\end{align}
depending on the sign of $f_4'(r)$, where
\begin{align*}
l_4(r)=\frac1{\sqrt{-h(r)}}\left(Hr+\frac{c}{r^2}\right).
\end{align*}
Here $c_{4}, \bar{c}_4, \bar{c}_4'$ are constants,
and $r_4, r_4'$ are fixed numbers in the domain of $f_4^{*}(r)$ and $f_4^{**}(r)$, respectively.

The function $l_{4}(r)$ should satisfy $l_{4}(r)>1$,
which implies $c_4>-8M^3H$ when $H\ge 0$ and $c_4>0$ when $H<0$.
In addition, we have $f_{4}'(r) \neq 0$ and will allow $f_{4}'(r)=\pm\infty$ at some point.
We will write $f_4(r)$ to denote both $f_4^{*}(r)$ and $f_4^{**}(r)$.
\end{prop}

From Proposition~\ref{prop6} and \ref{propinteriorII}, we know conditions $l_2(r)>1$ and $l_4(r)>1$
put restrictions on the domain of $f_{2}(r;H,c)$ and $f_{4}(r;H,c)$, respectively.
Remark that $c$ could be $c_2$ or $c_4$. Particularly, we have the equivalent conditions:
\begin{align*}
l_2(r)&=\frac1{\sqrt{-h(r)}}\left(-Hr-\frac{c}{r^2}\right)>1 \Leftrightarrow -Hr^3-r^{\frac32}(2M-r)^{\frac12}>c, \\
l_4(r)&=\frac1{\sqrt{-h(r)}}\left(Hr+\frac{c}{r^2}\right)>1 \Leftrightarrow -Hr^3+r^{\frac32}(2M-r)^{\frac12}<c.
\end{align*}
Define two functions $k_H(r)$ and $\tilde{k}_H(r)$ on $(0,2M)$ by
\begin{align}
k_{H}(r)&=-Hr^3-r^{\frac32}(2M-r)^{\frac12} \label{k2function} \\
\tilde{k}_{H}(r)&=-Hr^3+r^{\frac32}(2M-r)^{\frac12}, \label{ktilde}
\end{align}
then domains of $f_2(r;H,c)$ and $f_4(r;H,c)$ will be
\begin{align*}
&D_2=\{r\in(0,2M)|k_{H}(r)>c\}\cup\{r\in(0,2M)|k_{H}(r)=c\mbox{ and } f_{2}(r)\mbox{ is finite}\} \\
&D_4=\{r\in(0,2M)|\tilde{k}_{H}(r)<c\}\cup\{r\in(0,2M)|\tilde{k}_{H}(r)=c \mbox{ and } f_{4}(r) \mbox{ is finite}\}.
\end{align*}

\begin{figure}[h]
\centering
\psfrag{r}{$r$}
\psfrag{t}{$t$}
\psfrag{k}{$k_H(r)$}
\psfrag{K}{$\tilde{k}_H(r)$}
\psfrag{R}{\tiny$R_H$}
\psfrag{S}{\tiny$r_H$}
\psfrag{M}{\tiny$2M$}
\psfrag{H}{$-8M^3H$}
\psfrag{C}{$C_H$}
\psfrag{c}{$c_H$}
\psfrag{l}{\hspace*{-1mm}$L(r)=c$}
\includegraphics[height=50mm,width=65mm]{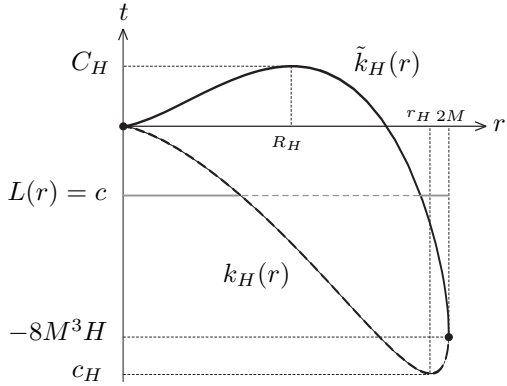}
\caption{Domain of $f_2(r)$ and $f_4(r)$.} \label{domainHplus02}
\end{figure}

It is easy to see domains $D_2$ and $D_4$ visually.
Take $H>0$ for example and cases $H=0$ and $H<0$ are similar. Figure~\ref{domainHplus02} illustrates the graphs of
$k_{H}(r)$ and $\tilde{k}_{H}(r)$, which form a loop in the $r$-$t$ plane.
Given any constant $c$, preimage of the line $L(r)=c$ below $k_{H}(r)$ belongs to $D_2$,
and preimage of the line above $\tilde{k}_{H}(r)$ belongs to $D_4$.

We still have to take care of the intersection of $L(r)=c$ and the loop, which may also belong to $D_2$ or $D_4$.
Let $c_H$ be the minimum value of $k_H(r)$ achieved at $r=r_H$ and let $C_H$ be the maximum value of $\tilde{k}_H(r)$ achieved at $r=R_H$.
After some analysis which can be found in \cite{LL1}, it turns out that when $c\in (c_H,C_H)$, intersections of $L(r)=c$ and $k_H(r)$ belong to $D_2$,
and intersections of $L(r)=c$ and $\tilde{k}_H(r)$ belong to $D_4$.
But when $c=c_H$ or $c=C_H$, their intersections do not belong to domains $D_2$ or $D_4$.

In conclusion, we have the following proposition on \SC hypersurfaces:
\begin{prop}{\rm\cite{LL1}}\label{propositioninterior}
\mbox{}
\begin{itemize}
\item[\rm(a)] If $c_H<c_2<\max(0,-8M^3H)$, then $f_{2}(r)$ is defined on $(0,r']$ or $[r'',2M)$ for some $r'$ and $r''$, which depend on $H$ and $c_2$.
When we take $r_2=r_2'=r' (\mbox{or } r'')$ and $\bar{c}_2=\bar{c}_2'$ in {\rm(\ref{f2positive})} and {\rm(\ref{f2negative})},
the union $\Sigma^2=(f_{2}^{*}(r;H,c_2,\bar{c}_2)\cup f_2^{**}(r;H,c_2,\bar{c}_2'),r,\theta,\phi)$ is a smooth \SC hypersurface in the Schwarzschild interior
{\rm(}region {\tt I$\!$I}{\rm)}.
\item[\rm(b)] If $\min(0,-8M^3H)<c_4<C_H$, then $f_{4}(r)$ is defined on $(0,r']$ or $[r'',2M)$ for some $r'$ and $r''$, which depend on $H$ and $c_4$.
When we take $r_4=r_4'=r' (\mbox{or } r'')$ and $\bar{c}_4=\bar{c}_4'$ in {\rm(\ref{f4positive})} and {\rm(\ref{f4negative})},
the union $\Sigma^4=(f_{4}^{*}(r;H,c_4,\bar{c}_4)\cup f_4^{**}(r;H,c_4,\bar{c}_4'),r,\theta,\phi)$ is a smooth \SC hypersurface in the Schwarzschild interior
{\rm(}region {\tt I$\!$I'}{\rm)}.
\end{itemize}
\end{prop}

\subsection{Smooth SS-CMC hypersurfaces} \label{CompleteSmoothSSCMC}
Given an \SC hypersurface in the Schwarzschild exterior or interior, as long as its domain is defined near $r=2M$,
we can discuss the behavior of \SC hypersurfaces near $r=2M$ and join two \SC hypersurfaces in different region smoothly in the Kruskal extension.

Here we only list results of \SC hypersurfaces needed in next section and refer to \cite{LL1} for all the other cases.
\begin{itemize}
\item[(a)] If $c_H<c_2<-8M^3H$, from Proposition~\ref{propositioninterior} (a), for every \SC hypersurface $\Sigma^2$ defined near $r=2M$,
then the spacelike condition is preserved as $r\to 2M^-$. Since $f_2^*\to\infty$ as $r\to 2M^-$,
the image of $\Sigma^2$ touches the interface of region {\tt I$\!$I} and {\tt I}.
We can take $\Sigma^1$ in region {\tt I} with $c_1=c_2$ and suitable $\bar{c}_1$ which is determined by $\bar{c}_2$ such that
$\Sigma^1\cup\Sigma^2$ is a smooth \SC hypersurface in the Kruskal extension.
Similarly, since $f_2^{**}\to-\infty$ as $r\to 2M^-$,
the image of $\Sigma^2$ touches the interface of region {\tt I$\!$I} and {\tt I'}.
We can take $\Sigma^3$ in region {\tt I'} with $c_3=c_2$ and suitable $\bar{c}_3$ which is determined by $\bar{c}_2$ such that
$\Sigma^1\cup\Sigma^2\cup\Sigma^3$ is a smooth \SC hypersurface in the Kruskal extension.
\item[(b)] If $-8M^3H<c_4<C_H$, from Proposition~\ref{propositioninterior} (b), for every \SC hypersurface $\Sigma^4$ defined near $r=2M$,
then the spacelike condition is preserved as $r\to 2M^-$. Since $f_4^{**}\to-\infty$ as $r\to 2M^-$,
the image of $\Sigma^4$ touches the interface of region {\tt I$\!$I'} and {\tt I}.
We can take $\Sigma^1$ in region {\tt I} with $c_1=c_4$ and suitable $\bar{c}_1$ which is determined by $\bar{c}_4$ such that
$\Sigma^1\cup\Sigma^4$ is a smooth \SC hypersurface in the Kruskal extension.
Similarly, since $f_4^{*}\to\infty$ as $r\to 2M^-$,
the image of $\Sigma^4$ touches the interface of region {\tt I$\!$I'} and {\tt I'}.
We can take $\Sigma^3$ in region {\tt I'} with $c_3=c_4$ and suitable $\bar{c}_3$ which is determined by $\bar{c}_4$ such that
$\Sigma^1\cup\Sigma^4\cup\Sigma^3$ is a smooth \SC hypersurface in the Kruskal extension.
\item[(c)] If $c_1=-8M^3H$, from Proposition~\ref{scmc3prop1}, every \SC hypersurface $\Sigma^1$ is defined near $r=2M$,
and the spacelike condition is preserved as $r\to 2M^+$. Since $f_1(r)$ tends to a finite value as $r\to 2M^+$,
we can extend $f_1(r)$ at $r=2M$, and the image of $\Sigma^1$ touches the origin of the Kruskal extension.
We can take $\Sigma^3$ in region {\tt I'} with $c_3=c_1$ and suitable $\bar{c}_3$ which is determined by $\bar{c}_1$ such that
$\Sigma^1\cup\Sigma^3$ is a smooth \SC hypersurface in the Kruskal extension.
\end{itemize}

The upshot of the characterization of \SC hypersurfaces is:
\begin{thm}{\rm\cite{LL1}}\label{thmall}
For $H\in\mathbb{R}$, all smooth \SC hypersurfaces and their behaviors in the Schwarzschild spacetimes or in the Kruskal extension
are completely characterized, by two constants $c$ and $\bar{c}$.
\end{thm}

\section{CMC foliation in the Kruskal extension} \label{CMCfoliation}
\subsection{A conjecture of CMC foliation}\label{conjecture}
A spherically symmetric hypersurface can be represented by a curve in the Kruskal plane,
and it is convenient to study the curve in null coordinates as $\Sigma=(U(s),V(s))$.
In such coordinates, the spacelike condition is equivalent to the tangent vector of $\Sigma$ being spacelike,
that is $V'(s)U'(s)>0$ from (\ref{SchMetric2}). Hence the curve can be written as a graph of $V(U)$ with
\begin{align*}
\frac{\mathrm{d}V}{\mathrm{d}U}=\frac{V'(s)}{U'(s)}>0.
\end{align*}
That is, $V(U)$ is a monotone increasing function.

\begin{defn}\label{defn1}
An \SC hypersurface $\Sigma=(U(s),V(s))$ in the Kruskal plane is called {\it $T$-axisymmetric} if
$\Sigma$ is symmetric with respect to the $T$-axis: $U+V=0$, or equivalently,
$\Sigma$ satisfies the following condition:
\begin{align}
\mbox{If } (U,V)\in\Sigma, \mbox{ then } (-V,-U)\in\Sigma. \tag{$*$} \label{Tsymm}
\end{align}
In the following, we use \TSC to represent $T$-axisymmetric \SC$\!\!$.
\end{defn}

In \cite{LL1}, we have constructed all smooth \SC hypersurfaces.
Every \SC hypersurface is characterized by two parameters $c$ and $\bar{c}$.
Now we consider \SC hypersurfaces with $c\in(c_H,C_H)$,
and cylindrical hypersurfaces $r=r_H$ in region {\tt I$\!$I} and $r=R_H$ in region {\tt I$\!$I'},
which correspond to $c=c_H$ and $c=C_H$, respectively.
Recall that $c_H<-8M^3H\leq 0<C_H$ when $H\geq 0$ and $c_H<0<-8M^3H<C_H$ when $H<0$.
If we put superscripts $+$ and $-$ on functions $k_H(r)$ in (\ref{k2function})
and $\tilde{k}_{H}(r)$ in (\ref{ktilde}) to represent their increasing and decreasing part, respectively,
we know that each $c\in(c_H,C_H), c\neq 0$ determines two families \SC hypersurfaces,
which can be distinguished by $k^+_H(r), k^-_H(r), \tilde{k}^+_H(r)$ or $\tilde{k}^-_H(r)$.
Denote their associated smooth \SC hypersurfaces by $\Sigma^+_{H,c,\bar{c}}$,
$\Sigma^-_{H,c,\bar{c}}, \tilde{\Sigma}^+_{H,c,\bar{c}}$ and $\tilde{\Sigma}^-_{H,c,\bar{c}}$, respectively.
For $c=0$, it determines one family of \SC hypersurfaces  belonging to $\tilde{\Sigma}^-_{H,c,\bar{c}}$ if $H\geq 0$,
or belonging to $\Sigma^+_{H,c,\bar{c}}$ if $H<0$.
The following Proposition shows the existence of \TSC hypersurfaces in each family.

\begin{prop}\label{propTaxisymm}
Among the family of \SC hypersurfaces $\Sigma^-_{H,c,\bar{c}}$, there exists a unique \TSC hypersurface $\Sigma^-_{H,c}$.
Same conclusion also holds for $\Sigma^+_{H,c,\bar{c}},\, \tilde{\Sigma}^+_{H,c,\bar{c}}$ and  $\tilde{\Sigma}^-_{H,c,\bar{c}}$.
Here $c\in(c_H,C_H)$ and $H\in\mathbb{R}$.
\end{prop}

\begin{proof}
For $\Sigma^-_{H,c,\bar{c}}$ with $c\in(c_H,0)$, hypersurfaces are in the Schwarzschild interior (region {\tt I$\!$I}).
Let $r_{H,c}^-$ be the solution of $k_H^-(r)=c$.
Choose $\bar{c}$ such that $f(r_{H,c}^-;H,c,\bar{c})=0$,
where the formula of $f$ is given in (\ref{f2positive}) or (\ref{f2negative})
\footnote{We will omit the subscripts $1,2,3,4$ of $f$ and $l$ as long as there is no confusing.}.
Denote this \SC hypersurface by $\Sigma^-_{H,c}$.
The hypersurface $\Sigma^-_{H,c}$ in $U+V\geq 0$ region corresponds to $f'(r)<0$ with domain $(0,r_{H,c}^-]$ in the Schwarzschild interior,
and it has nonnegative $t$-value:
$t=\int_{r_{H,c}^-}^rf_-'(x;c)\mathrm{d}x$, where $f'_-=\frac1{h}\sqrt{\frac{l^2}{l^2-1}}$ (see Proposition~\ref{prop6}, equation~(\ref{f2negative})).
From the table below (\ref{NullCoordinatesuv}), the corresponding $(U,V)$ coordinates are
\begin{align}
\begin{array}{l}
U(r;c)
=-\mathrm{e}^{-\frac1{4M}\left(t-r-2M\ln|r-2M|\right)}
=-\mathrm{e}^{-\frac1{4M}\left(\int^{r}_{r^-_{H,c}}f_-'(x;c)\mathrm{d}x-r-2M\ln|r-2M|\right)},\\
V(r;c)
=\mathrm{e}^{\frac1{4M}\left(t+r+2M\ln|r-2M|\right)}
=\mathrm{e}^{\frac1{4M}\left(\int^{r}_{r^-_{H,c}}f_-'(x;c)\mathrm{d}x+r+2M\ln|r-2M|\right)}.
\end{array} \label{case11}
\end{align}

On the other hand, the hypersurface $\Sigma^-_{H,c}$ in $U+V\leq 0$ region corresponds to $f'(r)>0$ with domain $(0,r_{H,c}^-]$ in the Schwarzschild interior,
and it has nonpositive $t$-value:
$t=\int_{r_{H,c}^-}^rf_+'(x;c)\mathrm{d}x$, where $f'_+=\frac1{-h}\sqrt{\frac{l^2}{l^2-1}}$.
The corresponding $(U,V)$ coordinates are
\begin{align}
\begin{array}{l}
U(r;c)=-\mathrm{e}^{-\frac1{4M}\left(t-r-2M\ln|r-2M|\right)}
=-\mathrm{e}^{-\frac1{4M}\left(\int^{r}_{r^-_{H,c}}f_+'(x;c)\mathrm{d}x-r-2M\ln|r-2M|\right)},\\
V(r;c)
=\mathrm{e}^{\frac1{4M}\left(t+r+2M\ln|r-2M|\right)}
=\mathrm{e}^{\frac1{4M}\left(\int^{r}_{r^-_{H,c}}f_+'(x;c)\mathrm{d}x+r+2M\ln|r-2M|\right)}.
\end{array} \label{case12}
\end{align}
Since $f_-'(x;c)=-f_+'(x;c)$ for all $x\in(0,r_{H,c}^-]$, (\ref{case11}) and (\ref{case12}) satisfy the condition (\ref{Tsymm}).

For $\Sigma^+_{H,c,\bar{c}}$ with $c\in(c_H,-8M^3H)$, these hypersurfaces pass through regions {\tt I, I$\!$I}, and {\tt I'}
(See discussion (a) before Theorem~\ref{thmall}).
Let $r_{H,c}^+$ be the solution of $k_H^+(r)=c$.
Choose $\bar{c}$ such that $f(r_{H,c}^+;H,c,\bar{c})=0$, where the formula of $f$ is given in (\ref{f2positive}) or (\ref{f2negative}).
Denote
\begin{align}
\bar{f}'(r)=\frac{r^4}{(Hr^3+c_1)^2+r^3(r-2M)-(Hr^3+c_1)\sqrt{(Hr^3+c_1)^2+r^3(r-2M)}},\label{barf}
\end{align}
then $\Sigma^+_{H,c}$ in $U+V\geq 0$ region satisfies $f'(r)=-\frac1{h(r)}+\bar{f}'(r)$.
Therefore, we have
\begin{align*}
t=f(r)
=&\int_{r_{H,c}^+}^r\left(-\frac1{h(x)}+\bar{f}'(x)\right)\mathrm{d}x \\
=&-r-2M\ln|r-2M|+r_{H,c}^++2M\ln|r_{H,c}^+-2M|+\int_{r_{H,c}^+}^r\bar{f}'(x)\mathrm{d}x.
\end{align*}
It leads to
{\small
\begin{align}
U(r;c)&=\left\{\begin{array}{ll}
\mathrm{e}^{-\frac{u}{4M}} & \mbox{in region {\tt I}}\\
-\mathrm{e}^{-\frac{u}{4M}} & \mbox{in region {\tt I$\!$I}}
\end{array}\right.
=\left\{\begin{array}{ll}
\mathrm{e}^{-\frac{1}{4M}(t-r-2M\ln|r-2M|)} & \mbox{in region {\tt I}}\\
-\mathrm{e}^{-\frac{1}{4M}(t-r-2M\ln|r-2M|)} & \mbox{in region {\tt I$\!$I}}
\end{array}\right. \notag \\
&=(r-2M)\,\mathrm{e}^{-\frac1{4M}\left(-2r+r_{H,c}^++2M\ln|r_{H,c}^+-2M|+\int_{r_{H,c}^+}^r\bar{f}'(x)\,\mathrm{d}x\right)} \notag \\
&=(r-2M)\,\mathrm{e}^{\frac1{4M}\left(2r-r_{H,c}^+-2M\ln|r_{H,c}^+-2M|-\int_{r_{H,c}^+}^r\bar{f}'(x)\,\mathrm{d}x\right)}. \notag \\
V(r;c)
&=\mathrm{e}^{\frac{v}{4M}}=\mathrm{e}^{\frac1{4M}(t+r+2M\ln|r-2M|)}
=\mathrm{e}^{\frac1{4M}\left(r_{H,c}^++2M\ln|r_{H,c}^+-2M|+\int_{r_{H,c}^+}^r\bar{f}'(x)\,\mathrm{d}x\right)}. \label{Vcoord}
\end{align}}

On the other hand, $\Sigma^+_{H,c}$ in $U+V\leq 0$ region satisfies $f'(r)=\frac1{h(r)}-\bar{f}'(r)$, and we have
\begin{align*}
U(r;c)&=-\mathrm{e}^{-\frac1{4M}\left(-r_{H,c}^+-2M\ln|r_{H,c}^+-2M|-\int_{r_{H,c}^+}^r\bar{f}'(x)\,\mathrm{d}x\right)} \\
      &=-\mathrm{e}^{\frac1{4M}\left(r_{H,c}^++2M\ln|r_{H,c}^+-2M|+\int_{r_{H,c}^+}^r\bar{f}'(x)\,\mathrm{d}x\right)} \\
V(r;c)&=-(r-2M)\,\mathrm{e}^{\frac1{4M}\left(2r-r_{H,c}^+-2M\ln|r_{H,c}^+-2M|-\int_{r_{H,c}^+}^r\bar{f}'(x)\,\mathrm{d}x\right)}.
\end{align*}
Hence $\Sigma^+_{H,c}$ satisfies condition (\ref{Tsymm}), and $\Sigma^+_{H,c}$ is a \TSC hypersurface.

When $c=-8M^3H$, $\tilde{\Sigma}_{H,c,\bar{c}}^-$ pass through region {\tt I} and {\tt I'} (See discussion (c) before Theorem~\ref{thmall}).
Choose $\bar{c}$ such that $f(2M;H,c=-8M^3H,\bar{c})=0$
in the Schwarzschild exterior,
and denote the hypersurface by $\tilde{\Sigma}^-_{H,-8M^3H}$.
The expression of $\tilde{\Sigma}^-_{H,-8M^3H}$ is
\begin{align*}
&\left\{
\begin{array}{l}
U(r)=\sqrt{r-2M}\,\mathrm{e}^{\frac1{4M}(-\int_{2M}^rf_1'(x)\,\mathrm{d}x+r)}\\
V(r)=\sqrt{r-2M}\,\mathrm{e}^{\frac1{4M}(\int_{2M}^rf_1'(x)\,\mathrm{d}x+r)}
\end{array}\right.\quad\mbox{in region {\tt I}},\\
&\left\{
\begin{array}{l}
U(r)=-\sqrt{r-2M}\,\mathrm{e}^{\frac1{4M}(-\int_{2M}^rf_3'(x)\,\mathrm{d}x+r)}\\
V(r)=-\sqrt{r-2M}\,\mathrm{e}^{\frac1{4M}(\int_{2M}^rf_3'(x)\,\mathrm{d}x+r)}
\end{array}\right.\quad\mbox{in region {\tt I'}},
\end{align*}
where
\begin{align*}
f'_3(r)=-f_{1}'(r)=-H\left(\frac{r}{r-2M}\right)^{\frac12}
\left(\frac{r(r^2+2Mr+4M^2)^2}{r^3+H^2(r-2M)(r^2+2Mr+4M^2)^2}\right)^{\frac12}.
\end{align*}
Hence $\tilde{\Sigma}^-_{H,-8M^3H}$ satisfies condition (\ref{Tsymm}),
and $\tilde{\Sigma}^-_{H,-8M^3H}$ is $T$-axisymmetric.

The other cases can be discussed similarly.
\end{proof}

With Definition~\ref{defn1} and Proposition~\ref{propTaxisymm},
we can rephrase the conjecture of Malec and \'{O} Murchadha in \cite{MO} as
\TSC hypersurfaces with $H$ fixed and $c$ varied from $c_H$ to $C_H$ form a foliation.
It will be explained in more details below.
We will discuss the case $H>0$ only. The case $H\leq 0$ is similar.

\begin{figure}[h]
\psfrag{A}{$0$}
\psfrag{B}{$c_H$}
\psfrag{C}{$C_H$}
\psfrag{E}{$-8M^3H$}
\psfrag{G}{$c=c_H$}
\psfrag{H}{$c=C_H$}
\psfrag{I}{$c=0$}
\psfrag{J}{$c=-8M^3H$}
\psfrag{r}{$r$}
\psfrag{M}{$R_H$}
\psfrag{N}{$r_H$}
\psfrag{P}{$2M$}
\psfrag{W}{$k_H^-$}
\psfrag{X}{$k_H^+$}
\psfrag{Y}{$\tilde{k}_H^-$}
\psfrag{Z}{$\tilde{k}_H^+$}
\psfrag{Q}{$T$}
\psfrag{S}{$X$}
\hspace*{-15mm}
\includegraphics[height=52mm, width=102mm]{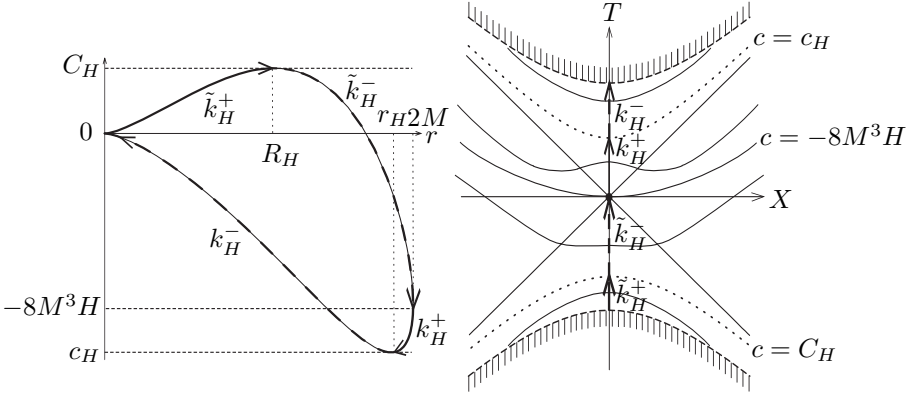}
\caption{Graphs of $k_H(r)=k_H^-\cup k_H^+$ and
$\tilde{k}_H(r)=\tilde{k}_H^+\cup\tilde{k}_H^-$, and a possible \TSC foliation for $H>0$.}
\label{CMCFOLIP4}
\end{figure}

Consider the graphs of $\tilde{k}_H^+\cup\tilde{k}_H^-\cup k_H^+\cup k_H^-$ as a loop,
which is illustrated in the left picture of Figure \ref{CMCFOLIP4}.
When we circle the loop clockwise from the origin,
the first piece is $\tilde{k}_H^+$ with $c\in(0,C_H)$ and their corresponding \TSC hypersurfaces are $\tilde{\Sigma}_{H,c}^+$.
Let $\tilde{r}_{H,c}^+$ be the solution of $\tilde{k}_H^+(r)=c$,
then the hypersurface $\tilde{\Sigma}_{H,c}^+$ intersects the $r$-axis at $r=\tilde{r}^+_{H,c}$ in the Schwarzschild interior (region {\tt I$\!$I'})
and from (\ref{trans}), $\tilde{\Sigma}^+_{H,c}$ intersects $T$-axis at
$T=-\sqrt{2M-\tilde{r}^+_{H,c}}\,\mathrm{e}^{\frac{\tilde{r}^+_{H,c}}{4M}}$ in the Kruskal plane.
The increasing property of $\tilde{k}_H^+$ implies that $\tilde{r}^+_{H,c}$ increases as $c$ increases,
so the $T$-intercept increases as well.

These \TSC hypersurfaces $\tilde{\Sigma}^+_{H,c}$ from $c=0$ to $c=C_H$
are conjectured to form a foliation between two hyperbolas $r=0$ and $r=R_H$,
where $R_H$ is the solution to $\tilde{k}_H^+(r)=C_H$.
When $c=C_H$, we take the cylindrical hypersurface $r=R_H$ to be the \TSC one, and call it $\tilde{\Sigma}_{H,C_H}^+$.
Denote $\tilde{\Sigma}_{H,0<c\leq C_H}^+=\{\tilde{\Sigma}_{H,c}^+|0<c\leq C_H\}$.

The second piece is $\tilde{k}_H^-$ with $c$ decreasing from $C_H$ to $-8M^3H$.
Let $\tilde{r}_{H,c}^-$ be the solution of $\tilde{k}_{H}^-(r)=c$,
then the corresponding \TSC hypersurface $\tilde{\Sigma}_{H,c}^-$ intersects the
$r$-axis at $r=\tilde{r}_{H,c}^-$ in the Schwarzschild interior (region {\tt I$\!$I'}),
and $\tilde{\Sigma}_{H,c}^-$ intersects $T$-axis at $T=-\sqrt{2M-\tilde{r}^-_{H,c}}\,\mathrm{e}^{\frac{\tilde{r}^-_{H,c}}{4M}}$ in the Kruskal plane.
The $T$-intercept increases when $c$ decreases.
When $c=-8M^3H$, the \TSC hypersurface $\tilde{\Sigma}^-_{H,-8M^3H}$ passes through the origin in the Kruskal plane.
Hence $\tilde{\Sigma}_{H,c}^-$ for $c$ ranging from $C_H$ to $-8M^3H$ lie between the hyperbola $r=R_H$ and $\tilde{\Sigma}^-_{H,-8M^3H}$,
and they are conjectured to form a foliation.
Denote $\tilde{\Sigma}^-_{H,C_H>c\geq -8M^3H}=\{\tilde{\Sigma}^-_{H,c}|C_H>c\geq -8M^3H\}$.

The third piece is $k_H^+$ with $c$ decreasing from $-8M^3H$ to $c_H$.
If $r_{H,c}^+$ is the solution of $k_{H}^+(r)=c$,
the \TSC hypersurface $\Sigma^+_{H,c}$ intersects $T$-axis at $T=\sqrt{2M-r^+_{H,c}}\,\mathrm{e}^{\frac{r^+_{H,c}}{4M}}$,
and the $T$-intercept increases when $c$ decreases.
The critical value $c_H$ determines a \TSC hypersurface $r=r_H$,
which is a cylindrical hypersurface, and we call it $\Sigma_{H,c_H}^+$.
Denote $\Sigma^+_{H,-8M^3H>c>c_H}=\{\Sigma_{H,c}^+|-8M^3H>c>c_H\}$.
It is conjectured that \TSC hypersurfaces in $\Sigma^+_{H,-8M^3H>c>c_H}$ form a foliation between $\tilde{\Sigma}^-_{H,-8M^3H}$
and the hyperbola $r=r_H$.

Finally, consider $k_H^-$ with $c$ increasing from $c_H$ to $0$.
Denote $\Sigma^-_{H,c_H\leq c<0}=\{\Sigma_{H,c}^-|c_H\leq c<0\}$.
These \TSC hypersurfaces $\Sigma_{H,c}^-$ lie between two hyperbolas $r=r_{H}$ and $r=0$ in region {\tt I$\!$I} with the
$T$-intercepts $T=\sqrt{2M-r^-_{H,c}}\ \mathrm{e}^{\frac{r^-_{H,c}}{4M}}$,
where $r_{H,c}^-$ is the solution of $k_H^-(r)=c$, and the $T$-intercept increases when $c$ increases.
So $\Sigma^-_{H,c_H\leq c<0}$ are conjectured to form a foliation between two hyperbolas $r=r_{H}$ and $r=0$.

As we move along the loop from $\tilde{k}_H^+$, $\tilde{k}_H^-$, $k_H^+$ to $k_H^-$,
the change of their $T$-intercepts is indicated in the right picture of Figure \ref{CMCFOLIP4}.
In these notions, Malec and \'{O} Murchadha's conjecture can be summarized as:

\begin{conj}[Malec, Edward and \'{O} Murchadha, Niall, \cite{MO}]\label{conj}
Given any constant mean curvature $H$, the Kruskal extension can be foliated by the \TSC hypersurfaces
$\{\Sigma_H\}=\tilde{\Sigma}_{H,0<c\leq C_H}^+\cup\tilde{\Sigma}_{H,C_H>c>-8M^3H}^-\cup\Sigma_{H,-8M^3H\geq c>c_H}^+\cup\Sigma_{H,c_H\leq c<0}^-$.
\end{conj}

\subsection{Criteria for a global CMC foliation} \label{criteria}
From the construction in section~\ref{conjecture},
each \TSC hypersurface in $\{\Sigma_H\}$ has different $T$-intercept.
The continuity property implies that $\{\Sigma_H\}$ forms a local foliation near the $T$-axis.
That is, for any two \TSC hypersurfaces in $\{\Sigma_H\}$,
there exists an open set $O$ in the Kruskal extension such that they are disjoint in $O$ near the $T$-axis.
Conjecture \ref{conj} claims that $\{\Sigma_H\}$ is a global \TSC foliation. In other words,
the open set can be taken as the Kruskal extension and $\{\Sigma_H\}$ covers the whole Kruskal extension.

To answer the conjecture, we will first derive criteria that hypersurfaces in $\{\Sigma_H\}$ are disjoint.
The $T$-axisymmetry implies that it suffices to consider the hypersurfaces in $U+V> 0$ region.
We restrict to this case from now on. First, consider the family of \TSC hypersurfaces $\Sigma_{H,c_H<c<0}^-$.
The $V$-coordinate of $\Sigma_{H,c}^-$ is given in (\ref{case11}):
\begin{align*}
V(r;c)
=\mathrm{e}^{\frac1{4M}\left(t+r+2M\ln|r-2M|\right)}
=\mathrm{e}^{\frac1{4M}\left(\int^{r}_{r^-_{H,c}}f_-'(x;c)\,\mathrm{d}x+r+2M\ln|r-2M|\right)},
\end{align*}
where
\begin{align}
\int^{r}_{r^-_{H,c}}f_-'(x;c)\,\mathrm{d}x
=\int_{r^-_{H,c}}^r\frac{x}{x-2M}\frac{-Hx^3-c}{\sqrt{(Hx^3+c)^2+x^3(x-2M)}}\,\mathrm{d}x. \label{Extendedatzero}
\end{align}
We remark that when $c\in(c_H,0)$, the integral can be extended to a finite value at $r=0$.
So $V(0;c)$ is defined,
and we can use $V(0;c)$ to derive a criterion to detect whether $\Sigma^-_{H,c}, c\in(c_H,0)$ are disjoint.
\begin{prop} \label{prop15}
If $\frac{\mathrm{d}V(0;c)}{\mathrm{d}c}\leq 0$
for all \TSC hypersurfaces in $\Sigma^-_{H,c_H<c<0}$, then the hypersurfaces are disjoint.
If $\frac{\mathrm{d}V(0;c)}{\mathrm{d}c}>0$ at
$c=c_0$, then $\Sigma_{H,c_0}^-$ will intersect some other hypersurface in $\Sigma_{H,c_H<c<0}^-$.
\end{prop}

\begin{proof}
For $0>c_1>c_2>c_H$, and for $r\in(0,2M)$ where $V(r;c_1)$ and $V(r;c_2)$ are defined, we have
\begin{align*}
\frac{V(r;c_1)}{V(r;c_2)}
=\frac{V(0;c_1)}{V(0;c_2)}\ \mathrm{e}^{\frac1{4M}\int_0^r(f'_-(x;c_1)-f'_-(x;c_2))\mathrm{d}x}.
\end{align*}
The condition $\frac{\mathrm{d}V(0;c)}{\mathrm{d}c}\leq 0$ implies $\frac{V(0;c_1)}{V(0;c_2)}\leq 1$.
Furthermore, from Proposition \ref{prop6},
we have
\begin{align*}
\frac{\mathrm{d}f'_{-}}{\mathrm{d}c}
=\frac{1}{h(r)}\frac{-1}{(l^2-1)^{\frac32}}\frac{\mathrm{d}l}{\mathrm{d}c}<0,
\end{align*}
so $f_{-}'(r;c_1)-f_{-}'(r;c_2)<0$, which implies
$\mathrm{e}^{\frac1{4M}\int_0^r(f'_-(x;c_1)-f'_-(x;c_2))\mathrm{d}x}<1$.
Therefore, $V(r;c_1)<V(r;c_2)$ for all $r$ is defined,
and hence $\Sigma_{H,c_H<c<0}^-$ are disjoint because an intersection point must have the same $r$ by (\ref{trans}).

If $\frac{\mathrm{d}V(0;c_0)}{\mathrm{d}c}>0$,
there exists $c_1>c_0$ and $\varepsilon>0$ such that $V(0;c_1)-V(0;c_0)=\varepsilon>0$,
which implies $\frac{V(0;c_1)}{V(0;c_0)}>1$.
Consider the function $F(r)=\frac{V(r;c_1)}{V(r;c_0)}$ defined on $r\in[0,r_{H,c_1}^-]$,
then
\begin{align*}
F(r_{H,c_1}^-)=\mathrm{e}^{-\frac1{4M}\int_{r_{H,c_0}}^{r_{H,c_1}}f_-'(x,c_0)\,\mathrm{d}x}<1
\end{align*}
because of $r_{H,c_1}^-<r_{H,c_0}^-$ and $f_-'(x,c_0)<0$.
By the intermediate value theorem, there exists $r_0\in(0,r_{H,c_1}^-)$ such that $F(r_0)=1$,
so $(U(r_0;c_0),V(r_0;c_0))$ is an intersection point.
\end{proof}

Similar arguments give a criterion for $\tilde{\Sigma}_{H,0<c<C_H}^+$.

\begin{prop}
If $\frac{\mathrm{d}U(0;c)}{\mathrm{d}c}\geq 0$ for all \TSC hypersurfaces in
$\tilde{\Sigma}^+_{H,0<c<C_H}$, then the hypersurfaces are disjoint.
If $\frac{\mathrm{d}U(0;c)}{\mathrm{d}c}<0$ at $c=c_0$,
then $\tilde{\Sigma}_{H,c_0}^+$ will intersect some other hypersurface in $\tilde{\Sigma}_{H,0<c<C_H}^+$.
\end{prop}

Next, we consider $\Sigma_{H,-8M^3H\geq c>c_H}^+$.
The $V$-coordinate of $\Sigma_{H,c}^+$ in $U+V> 0$ region is given in (\ref{Vcoord}):
\begin{align*}
V(r;c)
=\mathrm{e}^{\frac1{4M}(t+r+2M\ln|r-2M|)}
=\mathrm{e}^{\frac1{4M}\left(r_{H,c}^++2M\ln|r_{H,c}^+-2M|+\int_{r_{H,c}^+}^r\bar{f}'(x)\mathrm{d}x\right)}.
\end{align*}

\begin{prop}\label{prop17}
Given $r'>r_H$, if $\frac{\mathrm{d}V(r';c)}{\mathrm{d}c}\leq 0$ for all \TSC hypersurfaces in the family $\Sigma_{H,-8M^3H\geq c>c_H}^+$,
then the hypersurfaces are disjoint for all $r\in[r_H,r']$ defined.
If $\frac{\mathrm{d}V(r';c)}{\mathrm{d}c}>0$ at $c=c_0$,
then $\Sigma_{H,c_0}^+$ intersects some other hypersurface in $\Sigma_{H,-8M^3H\geq c>c_H}^+$.
\end{prop}

\begin{proof}
For $-8M^3H\geq c_1>c_2>c_H$ and $r<r'$, we have
\begin{align*}
\frac{V(r;c_1)}{V(r;c_2)}
=\frac{V(r';c_1)}{V(r';c_2)}\ \mathrm{e}^{\frac1{4M}\int_{r}^{r'}(\bar{f}'(x;c_2)-\bar{f}'(x;c_1))\mathrm{d}x}.
\end{align*}
The condition $\frac{\mathrm{d}V(r';c)}{\mathrm{d}c}\leq 0$ implies $\frac{V(r';c_1)}{V(r';c_2)}\leq 1$.
One can check $\frac{\mathrm{d}\bar{f}'}{\mathrm{d}c}>0$
and it gives $\bar{f}'(r;c_1)>\bar{f}'(r;c_2)$, so
$\mathrm{e}^{\frac1{4M}\int_{r}^{r'}(\bar{f}'(x;c_2)-\bar{f}'(x;c_1))\mathrm{d}x}<1$.
Therefore, $\frac{V(r;c_1)}{V(r;c_2)}<1$ for all $r$ defined.
Since $UV=(r-2M)\mathrm{e}^{\frac{r}{2M}}$, it shows that  $\Sigma_{H,-8M^3H\geq c>c_H}^+$ are disjoint.

If $\frac{\mathrm{d}V(r';c_0)}{\mathrm{d}c}>0$, there exists $c_1$ with $c_1<c_0$
and $\varepsilon>0$ such that $V(r';c_0)-V(r';c_1)=\varepsilon>0$,
which implies $\frac{V(r';c_0)}{V(r';c_1)}>1$.
Let $F(r)=\frac{V(r;c_0)}{V(r;c_1)}$, which is defined on $r\in[r_{H,c_0}^+,\infty)$.
Since
\begin{align*}
F(r_{H,c_0}^+)=\frac{V(r_{H,c_0}^+;c_0)}{V(r_{H,c_0}^+;c_1)}<\frac{V(r_{H,c_0}^+;c_0)}{V(r_{H,c_1}^+;c_1)}<1,
\end{align*}
by the intermediate value theorem, there is $r_0\in(r_{H,c_0}^+,r')$ such that $F(r_0)=1$,
and $(U(r_0;c_0),V(r_0;c_0))$ is the intersection point.
\end{proof}

We also have a criterion for $\tilde{\Sigma}_{H,C_H>c>-8M^3H}^-$.
\begin{prop}
Given $r'>R_H$, if $\frac{\mathrm{d}U(r';c)}{\mathrm{d}c}\geq 0$ for all \TSC hypersurfaces in the family
$\tilde{\Sigma}_{H,C_H>c>-8M^3H}^-$, then the hypersurfaces are disjoint for $r\in[R_H,r']$ defined.
If $\frac{\mathrm{d}U(r';c)}{\mathrm{d}c}<0$ at $c=c_0$,
then $\tilde{\Sigma}_{H,c_0}^-$ intersects some other hypersurface in $\tilde{\Sigma}_{H,C_H>c>-8M^3H}^-$.
\end{prop}

For \TSC hypersurfaces in families $\Sigma_{H,c_H<c<0}^-$ and $\Sigma_{H,-8M^3H\geq c>c_H}^+$,
$V$-coordinates are positive in $U+V> 0$ region,
so the criteria for disjoint hypersurfaces can be replaced by $\frac{\mathrm{d}\ln V}{\mathrm{d}c}\leq 0$.
Similarly,
for \TSC hypersurfaces in $\tilde{\Sigma}_{H,0<c<C_H}^+$ and $\tilde{\Sigma}_{H,C_H>c>-8M^3H}^-$,
$U$-coordinates are positive in $U+V> 0$ region,
so the criteria for disjoint hypersurfaces can be replaced by $\frac{\mathrm{d}\ln U}{\mathrm{d}c}\geq 0$.

Finally, we remark that two families $\Sigma_{H,-8M^3H\geq c>c_H}^+$ and $\Sigma_{H,c_H<c<0}^-$
must be disjoint because the cylindrical hypersurface $r=r_H$ is a barrier between them.
Similarly, $r=R_H$ is a barrier between $\tilde{\Sigma}_{H,0<c<C_H}^+$ and $\tilde{\Sigma}_{H,C_H>c>-8M^3H}^-$.

\subsection{CMC foliation for $\Sigma^-_{H,c_H\leq c<0}$ in region {\tt I$\!$I} and
$\tilde{\Sigma}^+_{H,0<c\leq C_H}$ in region {\tt I$\!$I'}} \label{CMCFinterior}
In the following,
we use $R=R(H,c)$ to denote $r_{H,c}^-$ ($r_{H,c}^+, \tilde{r}_{H,c}^+$, or $\tilde{r}_{H,c}^-$)
in convenience when it does not cause confusion.
\begin{prop}
For all $H\in\mathbb{R}$, each \TSC hypersurface in $\Sigma_{H,c_H<c<0}^-$ satisfies $\frac{\mathrm{d}\ln V(0;c)}{\mathrm{d}c}\leq 0$
so that they are disjoint.
\end{prop}

\begin{proof}
From the calculation in the Appendix \ref{app71}, we have the formula (\ref{VformulaII0}):
\begin{align*}
\frac{\mathrm{d}\ln V(0;c(R))}{\mathrm{d}c}
=-\frac1{4MJ(R)\sqrt{-h(R)}}\left(\int_{0}^{R}\frac{H\cdot F(x,R)+G(x,R)}{(R-x)^{\frac12}(P(x,R))^{\frac32}}\,\mathrm{d}x-1\right),
\end{align*}
where $F(x,R)$ and $G(x,R)$ are as (\ref{Ffunction}) and (\ref{Gfunction}):
\begin{align*}
F(x,R)&=x^2(-3x^2(x+R-2M)+(2x-3M)(x^2+Rx+R^2)) \\
&=x^2((3M-x)(x^2-R^2)+xR(R-3M-x))\quad\mbox{and} \\
G(x,R)&=x^2\sqrt{-h(R)}(x(R-3M)+R(x-3M)).
\end{align*}
Because both $F(x,R)$ and $G(x,R)$ are negative functions on $x\in[0,R]$, and $J(R)=-3HR^{\frac32}(2M-R)+(2R-3M)<0$,
we get $\frac{\mathrm{d}\ln V(0;c)}{\mathrm{d}c}\leq 0$ if $H\geq 0$.

If $H<0$, we can still show that $H\cdot F(x,R)+G(x,R)\leq 0$ for all $x\in[0,R]$,
which also implies $\frac{\mathrm{d}\ln V(0;c)}{\mathrm{d}c}\leq 0$.
The proof goes as below.
Define $a=\frac{2M}{r_H}, b=\frac{2M}{R}$, and $z=\frac{x}{R}$.
We have relations $h(R)=1-\frac{2M}R=1-b$, and from (\ref{HRformula})
\begin{align*}
HR=\frac{a}b\frac1{\sqrt{a-1}}\left(\frac{4-3a}6\right).
\end{align*}
Hence $R<r_H$, $H<0$, and $x\in[0,R]$ imply
$b>a>\frac43$ and $0\leq z\leq 1$.
Furthermore, we have
\begin{align*}
&\ H\cdot F(x,R) \\
=&\ \frac{a}b\frac{R^4z^2}{\sqrt{a-1}}\left(\frac{4-3a}6\right)\left(-3z^2(z+1-b)+\left(2z-\frac32b\right)(z^2+z+1)\right) \\
=&\ \frac{a}b\frac{R^4z^2}{\sqrt{a-1}}\left(\frac{3a-4}6\right)\left(z(z-1)(z+2)+\frac32b\left(-\left(z-\frac12\right)^2+\frac54\right)\right).
\end{align*}
The term $z(z-1)(z+2)$ is negative on $z\in[0,1]$, so
\begin{align*}
H\cdot F(x,R)\leq\frac{a}b\frac{R^4z^2}{\sqrt{a-1}}\left(\frac{3a-4}6\right)(3b)=\frac{R^4z^2a}{\sqrt{a-1}}\left(\frac{3a-4}2\right).
\end{align*}
For $G(x,R)$, because $b>\frac43$, we have
\begin{align*}
G(x,R)&=R^4z^2\sqrt{b-1}\left(z\left(1-\frac32b\right)+\left(z-\frac32b\right)\right) \\
&=R^4z^2\sqrt{b-1}\left(z\left(2-\frac32b\right)-\frac32b\right) \\
&\leq-\frac{3R^4z^2b\sqrt{b-1}}{2}
\leq-\frac{3R^4z^2a\sqrt{a-1}}2.
\end{align*}
Hence
\begin{align*}
H\cdot F(x,R)+G(x,R)
\leq&\ \frac{R^4z^2}{\sqrt{a-1}}\left(\frac{a(3a-4)}{2}-\frac{3a(a-1)}{2}\right) \\
=&-\frac{a}2\frac{R^4z^2}{\sqrt{a-1}}<0.
\end{align*}
\end{proof}

To show $\Sigma_{H,c_H<c<0}^-$ forms a foliation between $r=0$ and $r=r_H$ in the Kruskal extension {\tt I$\!$I},
we still need to prove that every point can be covered by some \TSC in $\Sigma_{H,c_H<c<0}^-$.
As explained after (\ref{Extendedatzero}), $f(0;H,c)=\lim\limits_{r\to 0}f(r;H,c)$ is defined.
We will estimate $f(0;H,c)$ when $c\to 0$ or $c\to c_H$.
\begin{prop}
For \TSC hypersurfaces in $\Sigma_{H,c_H<c<0}^-$, we have
\begin{align*}
\lim_{c\to c_{H}}f(0;H,c)=\infty,
\quad\mbox{and}\quad\lim_{c\to 0}f(0;H,c)=0.
\end{align*}
\end{prop}

\begin{proof}
First, we know
\begin{align*}
f(0;H,c)
\geq&\int_{\frac{r_H}2}^{R}\frac{r}{r-2M}\frac{Hr^3+c(R)}{\sqrt{(Hr^3+c(R))^2+r^3(r-2M)}}\,\mathrm{d}r\ \ \mbox{ for all }R>\frac{r_H}2,
\end{align*}
where $c(R)=-HR^3-R^{\frac32}(2M-R)^{\frac12}$.
Consider
\begin{align*}
&\ (Hr^3+c(R))^2+r^3(r-2M) \\
=&\ (R-r)P_1(r,R,\mbox{deg}=5) \\
=&\ (R-r)^2P_2(r,R,\mbox{deg}=4)+(R-r)(2R^2)
(3HR^{\frac32}(2M-R)^{\frac12}+3M-2R),
\end{align*}
where $P_1$ and $P_2$ are polynomials with respect to $r$.
Because $(Hr^3+c(R))^2+r^3(r-2M)\geq 0$, we know $P_1(r,R,\mbox{deg}=5)>0$ for all $r\in[0,R]$.
We also have $3HR^{\frac32}(2M-R)^{\frac12}+3M-2R>0$ for $R<r_H$.
Let
\begin{align*}
Q_1(R)&=\max_{r\in[\frac{r_H}2,R]}P_2(r,R,\mbox{deg}=4), \\
Q_2(R)&=2R^2(3HR^{\frac32}(2M-R)^{\frac12}+3M-2R)>0, \\
m(R)&=\min_{r\in[\frac{r_H}2,R]}\frac{r(Hr^3+c(R))}{r-2M},\quad\mbox{and}\quad
m=\min_{R\in[\frac{r_H}2,r_H]}m(R)>0.
\end{align*}
We have
\begin{align*}
f(0;H,c)
\geq&\ m\int_{\frac{r_H}2}^{R}\frac{1}{\sqrt{(R-r)^2Q_1(R)+(R-r)Q_2(R)}}\,\mathrm{d}r.
\end{align*}
The quantity $Q_2\to 0$ and $Q_1$ bounded as $R\to r_H$ ($c\to c_H)$ imply $f(0;H,c)\to\infty$ as $c\to c_H$.
This can also be seen by direct computation that the integration gives
\begin{align*}
\frac{m}{\sqrt{Q_1}}\ln\left|\frac{2Q_1}{Q_2}\left(R-\frac{r_H}{2}+\frac{Q_2}{2Q_1}\right)
+\sqrt{\left(\frac{2Q_1}{Q_2}\left(R-\frac{r_H}{2}+\frac{Q_2}{2Q_1}\right)\right)^2-1}\right|.
\end{align*}

When $c$ tends to $0$, we split $f(0;H,c)$ into two parts
\begin{align*}
f(0;H,c)
=&\int_{0}^{\frac{R}{2}}\frac{1}{-h(r)}\sqrt{\frac{l^2(r;c)}{l^2(r;c)-1}}\,\mathrm{d}r
+\int_{\frac{R}2}^{R}\frac{1}{-h(r)}\sqrt{\frac{l^2(r;c)}{l^2(r;c)-1}}\,\mathrm{d}r \\
=&\ \mbox{(I)}+\mbox{(II)}.
\end{align*}
The square root term of the first part has a maximum value at $r=r_*=r_*(R)$, so
{\small
\begin{align*}
\mbox{(I)}
\leq\sqrt{\frac{l^2(r_*;c)}{l^2(r_*;c)-1}}\int_{0}^{\frac{R}{2}}
\frac{1}{-h(r)}\,\mathrm{d}r
=\sqrt{\frac{l^2(r_*;c)}{l^2(r_*;c)-1}}\left(-\frac{R}{2}+2M\ln\left|\frac{2M}{2M-\frac{R}{2}}\right|\right).
\end{align*}}
For the second part,
\begin{align*}
\mbox{(II)}
=&\ \int_{\frac{R}2}^{R}\frac{1}{-h(r)}\frac{-Hr^3-c}{\sqrt{(Hr^3+c)^2+r^3(r-2M)}}\,\mathrm{d}r \\
\leq&\ Q_3(R)\int_{\frac{R}2}^{R}\frac{1}{\sqrt{R-r}}\,\mathrm{d}r
=Q_3(R)\sqrt{2R},
\end{align*}
where $Q_3(R)= \max_{r\in[\frac{R}2,R]}\frac{1}{-h(r)}\frac{-Hr^3-c(R)}{\sqrt{P_1(r,R,\mbox{\footnotesize deg}=5)}}$.
Hence we have
\begin{align*}
0\leq f(0;H,c)
\leq
\sqrt{\frac{l^2(r_*;c)}{l^2(r_*;c)-1}}
\left(-\frac{R}{2}+2M\ln\left|\frac{2M}{2M-\frac{R}{2}}\right|\right)
+Q_3(R)\sqrt{2R}.
\end{align*}
As $c\to 0$, we have $R\to 0$ and $l^2(r_*;c)$ bounded away from $1$ as well as $Q_3(R)$ being bounded.
So right hand side of the above inequality tends to zero when  $R \to 0$, and it gives $f(0;H,c)\to 0$ as $c\to 0$.
\end{proof}

Since $\lim\limits_{c\to c_H}f(0,H,c)=\infty$ and $\lim\limits_{c\to 0}f(0;H,c)=0$,
for any level set $t=t_0>0$ there is $c_0$ such that $t_0=f(0;H,c_0)$ and $f(0;H,c)>t_0$ for all $c\in(c_H,c_0)$.
For given $H\in\mathbb{R}$ and $c\in(c_H,0)$, because $f(r;H,c)>t_0>0$ and $f(R;H,c)=0$, there exists $r=r(c,t_0)$ such that $f(r;H,c)=t_0$.

\begin{prop}
The \TSC family $\Sigma_{H,c_H<c<0}^-$
pointwise converges to the cylindrical hypersurface $r=r_H$ as $c\to c_H^+$.
\end{prop}

\begin{proof}
For \SC hypersurfaces
\begin{align*}
\left(f(r;H,c)
=\int_r^R\frac{x}{x-2M}\frac{Hx^3+c}{\sqrt{(Hx^3+c)^2+x^3(x-2M)}}\,\mathrm{d}x,r,\theta,\phi\right),
\end{align*}
$f(r;H,c)$ is a continuous function with respect to the parameter $c$.
To prove the proposition, it suffices to show that
\begin{align*}
\lim_{c\to c_H}r(c,t_0)=r_H.
\end{align*}

Fix $t_0>0$, for any $c\in(c_H,c_0)$ there exists $r_0=r_0(R)$ such that
\begin{align*}
t_0=\int_{r_0(R)}^{R}\frac{r}{r-2M}\frac{Hr^3+c}{\sqrt{(Hr^3+c)^2+r^3(r-2M)}}\,\mathrm{d}r.
\end{align*}
Note that $c=-HR^3-R^{\frac32}(2M-R)^{\frac12}$, where $c$ and $R$ can determine each other uniquely in the family $\Sigma_{H,c_H<c<0}^-$
($\tilde{\Sigma}_{H,0<c<C_H}^+, \tilde{\Sigma}_{H,C_H>c>-8M^3H}^-$, and $\Sigma_{H,-8M^3H\geq c>c_H}^+$).
Hence we can use $R$ as parameter instead.
Letting $c$ tend to $c_H$,
if $r_0(R)\not\to r_H$,
then right hand side will be unbounded, and it contradicts to the finite value of left hand side.
Hence we have $r_0(R)\to r_H$.
\end{proof}

Combining all the results above gives the following theorem.
\begin{thm}\label{CMCfoliation1}
For all $H\in\mathbb{R}$,
the \TSC family $\Sigma^-_{H,c_H<c<0}$ forms a foliation between two cylindrical hypersurfaces $r=0$ and $r=r_H$ in the Kruskal extension {\tt I$\!$I}.
\end{thm}

The same arguments lead to the \C foliation for $\tilde{\Sigma}_{H,0<c\leq C_H}^+$:
\begin{thm}\label{CMCfoliation2}
For all $H\in\mathbb{R}$,
the \TSC family $\tilde{\Sigma}_{H,0<c\leq C_H}^+$ forms a foliation between two cylindrical hypersurfaces $r=0$ and $r=R_H$ in the Kruskal extension {\tt I$\!$I'}.
\end{thm}

\subsection{CMC foliation for $\Sigma_{H,-8M^3H>c>c_H}^+$ and $\tilde{\Sigma}_{H,C_H>c>-8M^3H}^-$}\label{CMCfoliationII}
\begin{prop}
There exists a constant $C>0$ such that for any given $H\geq-C$,
$\frac{\mathrm{d}\ln V(2M;c)}{\mathrm{d}c}\leq 0$, which means hypersurfaces in
$\Sigma_{H,-8M^3H>c>c_H}^+$ are disjoint in region {\tt I$\!$I} when $H\geq -C$.
\end{prop}
\begin{proof}
We refer to the Appendix \ref{app72} for the calculation of $\frac{\mathrm{d}\ln V(2M;c(R))}{\mathrm{d}c}$,
which gives
\begin{align*}
\frac{\mathrm{d}\ln V(2M;c(R))}{\mathrm{d}c}
=\frac{1}{4MJ(R)\sqrt{-h(R)}}
\left(\int_{R}^{2M}
\frac{H\cdot F(x,R)+G(x,R)}
{(x-R)^{\frac12}(P(x,R))^{\frac32}}\,\mathrm{d}x-1\right)
\end{align*}
in (\ref{criterion2M}).
Note that $c(R)=-HR^3-R^{\frac32}(2M-R)^{\frac12}$, $J(R)>0$ is given in (\ref{JR}) and $F(x,R), G(x,R)$ are in (\ref{Ffunction2}), (\ref{Gfunction}), respectively.
Function $G(x,R)$ is negative on $[R,2M]$ because $R<2M$ and $x<2M$.

Next we show that $F(x,R)$ is negative on $[R,2M]$ when $H\geq 0$.
By the change of variables $b=\frac{2M}R$ and $z=\frac{x}R$, $F(x,R)$ can be expressed as
\begin{align*}
R^5z^2\left(-z^3+\left(\frac{3b}2-1\right)z^2+\left(2-\frac{3b}2\right)z-\frac{3b}2\right).
\end{align*}
Let $\bar{F}(z,b)=-z^3+\left(\frac{3b}2-1\right)z^2+\left(2-\frac{3b}2\right)z-\frac{3b}2$.
Note that $H\geq 0$ implies $b<\frac43$.
Since the coefficient of the highest order term of $\bar{F}$ is negative and $\bar{F}(\frac{3b-4}{2},b)=\bar{F}(0,b)=\bar{F}(1,b)=-\frac{3b}2<0$,
we have $\bar{F}(z,b)<0$ for $z\in[1,b]$ with $b<\frac43$, as $\frac{3b-4}{2}<0<1$ in this case.
So $F(x,R)$ is negative on $[R,2M]$ when $H\geq 0$ and thus $\frac{\mathrm{d}\ln V(2M;c(R))}{\mathrm{d}c}<0$ when $H\geq 0$.
By the continuity, we get $\frac{\mathrm{d}\ln V(2M;c(R))}{\mathrm{d}c}\leq 0$ for $H\geq -C$.
\end{proof}

\begin{prop}\label{prop23}
For \TSC hypersurfaces in $\Sigma_{H,-8M^3H>c>c_H}^+$ in region {\tt I$\!$I}, we have
\begin{align*}
\lim_{c\to c_H}V(2M;c)=\infty,\quad\mbox{and}\quad\lim_{c\to -8M^3H}V(2M;c)=0.
\end{align*}
\end{prop}

\begin{proof}
Recall that from (\ref{Vcoord}),
\begin{align*}
V(r;c)
=\mathrm{e}^{\frac1{4M}\left(r_{H,c}^++2M\ln|r_{H,c}^+-2M|+\int_{r_{H,c}^+}^r\bar{f}'(x)\,\mathrm{d}x\right)},
\end{align*}
where
\begin{align*}
\bar{f}'(r)=\frac{r^4}{(Hr^3+c)^2+r^3(r-2M)-(Hr^3+c)\sqrt{(Hr^3+c)^2+r^3(r-2M)}}.
\end{align*}
is given in (\ref{barf}). Notice that
{\small
\begin{align}
&\ (Hr^3+c(R))^2+r^3(r-2M) \notag \\
=&\ (r-R)^2P_3(r,R,\mbox{deg}=4)+(r-R)(2R^2)\left(-3HR^{\frac32}(2M-R)^{\frac12}-3M+2R\right) \label{eqn222}
\end{align}}
with $-3HR^{\frac32}(2M-R)^{\frac12}-3M+2R>0$ for $R>r_H$.
Let
\begin{align*}
Q_4(R)&=\max_{r\in[R,2M]}P_3(r,R,\mbox{deg}=4), \\
Q_5(R)&=2R^2(-3HR^{\frac32}(2M-R)^{\frac12}-3M+2R)>0, \\
m(R)&=\min_{r\in[R,2M]}\frac{r^4}{\sqrt{(Hr^3+c)^2+r^3(r-2M)}-(Hr^3+c)},\quad\mbox{and}\\
m&=\min_{R\in[r_H,2M]}m(R)>0.
\end{align*}
We have
\begin{align*}
\int_R^{2M}\bar{f}'(r,c(R))\,\mathrm{d}r
\geq&\ m\int_{R}^{2M}\frac{1}{\sqrt{(r-R)^2Q_4(R)+(r-R)Q_5(R)}}\,\mathrm{d}r.
\end{align*}
The quantity $Q_5\to 0$ and $Q_4$ bounded as $R\to r_H$ ($c\to c_H)$ imply $V(2M;c)\to\infty$ as $c\to c_H$.

Now we look at the case $c\to 0$, that is $R\to 2M$.
To study the limit $\lim\limits_{R\to 2M}\int_R^{2M}\bar{f}'(r,c(R))\,\mathrm{d}r$,
we need to estimate the denominator of $\bar{f}'(r)$. From (\ref{eqn222}) and
{\small
\begin{align*}
&\lim_{R\to 2M}\left(\min_{r\in[R,2M]}(r-R)P_3(r,R,\mathrm{deg}=4)+(2R^2)\left(-3HR^{\frac32}(2M-R)^{\frac12}-3M+2R\right)\right)\\
&=8M^3, \\
&\lim_{R\to 2M}\min_{r\in[R,2M]}\frac{-(Hr^3+c)}{(2M-R)^{\frac12}}
=\lim_{R\to 2M}\min_{r\in[R,2M]}\frac{-Hr^3+HR^3+R^{\frac32}(2M-R)^{\frac12}}{(2M-R)^{\frac12}} \\
&=(2M)^{\frac32},
\end{align*}}
for $R$ close to $2M$, we have the following estimate
{\small
\begin{align*}
\int_{R}^{2M}\bar{f}'(r,c(R))\mathrm{d}r\leq&\ \int_R^{2M}\frac{(2M)^4}{M^3(r-R)+M^3(2M-R)^{\frac12}(r-R)^{\frac12}}\,\mathrm{d}r \\
\leq&\ 16M\int_R^{2M}\frac{1}{(r-R)+(2M-R)^{\frac12}(r-R)^{\frac12}}\,\mathrm{d}r=32M\ln2.
\end{align*}}
Hence
\begin{align*}
\lim_{c\to 0}V(2M;c)
=\lim_{R\to 2M}\sqrt{2M-R}\ \mathrm{e}^{\frac1{4M}\left(R+\int_{R}^{2M}\bar{f}'(x)\,\mathrm{d}x\right)}=0.
\end{align*}
\end{proof}

\begin{prop}
The \TSC family $\Sigma_{H,-8M^3H>c>c_H}$ in region {\tt I$\!$I} pointwise converges to the cylindrical hypersurface $r=r_H$
as $c\to c_H$.
\end{prop}

\begin{proof}
Given $t_0>0$ and $c\in(c_H,-8M^3H)$, there uniquely exists $r_0=r_0(R)\in[R,2M)$ such that
\begin{align*}
t_0=\int_{R}^{r_0(R)}f'(x,c(R))\,\mathrm{d}x=\int_{R}^{r_0(R)}\frac{x}{x-2M}\frac{Hx^3+c}{\sqrt{(Hx^3+c)^2+x^3(x-2M)}}\,\mathrm{d}x.
\end{align*}
Letting $c$ tend to $c_H$, if $r_0(R)\not\to r_H$ then right hand side will be unbounded,
and it contradicts to the finite value of left hand side. Hence we have $r_0(R)\to r_H$.
\end{proof}

The case $\tilde{\Sigma}_{H,C_H>c>-8M^3H}^-$ can be treated similarly, and we can conclude that
\begin{thm}\label{localCMCfoliation34}
There exists a constant $C>0$ such that for any given $H\geq -C$, the \TSC family $\Sigma_{H,-8M^3H\geq c>c_H}^+$ forms a
foliation in region {\tt I$\!$I}, and for any given $H\leq C$, the \TSC family $\tilde{\Sigma}_{H,C_H>c>-8M^3H}^-$
forms a foliation in region {\tt I$\!$I'}.
\end{thm}

\subsection{Maximal hypersurfaces foliation in the Kruskal extension}\label{Maxfoliation}
In this subsection, we will show that $T$-axisymmetric spacelike spherically symmetric maximal hypersurfaces form a foliation in the whole Kruskal extension.
From Theorem \ref{CMCfoliation1} and \ref{localCMCfoliation34}, we know that $\{\Sigma_{H=0}\}$ forms a foliation in region {\tt I$\!$I} and {\tt I$\!$I'}.
Using the above method, we can also prove that $\{\Sigma_{H=0}\}$ forms a foliation in region {\tt I} and {\tt I'}.
Thus it forms a foliation in the whole Kruskal extension.
This reproves the result of Beig and \'{O} Murchadha in \cite{BO}.

When $H=0$, from (\ref{k2function}) and (\ref{ktilde}), we have $k_{H}(r)=-\tilde{k}_{H}(r)$.
Since $C_{H}$ is the maximum value of $\tilde{k}_H(r)$ and $c_{H}$ is the minimum value of $k_{H}(r)$, we get $C_{H}=-c_{H}$.
So two hypersurfaces $\Sigma^+_{H=0,c}$ and $\tilde{\Sigma}^-_{H=0,-c}, c\in(c_H,0)$ are symmetric about the $X$-axis in the Kruskal extension.
To prove the maximal hypersurfaces foliation, it suffices to show the case of $\Sigma^+_{H=0,0\geq c>c_H}$.

Remark that Proposition~\ref{prop17} implies that in $\Sigma_{H=0,0>c>c_H}^+$,
if the limit $\lim\limits_{r\to\infty}\frac{\mathrm{d}\ln V(r,c)}{\mathrm{d}c}\leq 0$, then hypersurfaces are disjoint.
Referring to the computation of $\frac{\mathrm{d}\ln V(r,c)}{\mathrm{d}c}$ in Appendix \ref{app72} and \ref{app73}, we put $H=0$ in (\ref{criterion4}).
In this case, $a=\frac43$, and from (\ref{changevariable}) and (\ref{bdyterm}) we get
{\small
\begin{align}
&\ 4M(2R-3M)\frac{\mathrm{d}\ln V(r;c(R))}{\mathrm{d}c} \notag\\
=&\int_1^{\frac{r}R}\frac{z^2\left(\left(2-\frac32b\right)z-\frac32b\right)}{(z-1)^{\frac12}\left(z^3+z^2+z+1-b(z^2+z+1)\right)^{\frac32}}\,\mathrm{d}z
-\left.\frac{1}{\sqrt{(b-1)+z^3(z-b)}}\right|_{z=\frac{r}R}. \label{maxfoliationcase}
\end{align}}
Consider the limit of (\ref{maxfoliationcase}) as $r$ tends to infinity and let $y=z-1$:
\begin{align}
&\ \int_1^\infty\frac{z^2\left(\left(2-\frac32b\right)z-\frac32b\right)}{(z-1)^{\frac12}\left((z^3+z^2+z+1)-b(z^2+z+1)\right)^{\frac32}}\,\mathrm{d}z \notag\\
=&\ \int_0^\infty\frac{\left(2-\frac32b\right)y^3+(6-6b)y^2
+\left(6-\frac{15}2b\right)y+(2-3b)}{y^{\frac12}\left(y^3+(4-b)y^2+(6-3b)y+(4-3b)\right)^{\frac32}}\,\mathrm{d}y. \label{maxestimate}
\end{align}
In this case, $1\leq b\leq\frac43$, and the denominator can be bounded by
\begin{align*}
y(y+1)^2\leq y^3+(4-b)y^2+(6-3b)y+(4-3b)\leq (y+1)^3.
\end{align*}
Hence (\ref{maxestimate}) has the following estimate:
{\small
\begin{align*}
&\ \int_0^\infty\frac{\left(2-\frac32b\right)y^3+(6-6b)y^2
+\left(6-\frac{15}2b\right)y+(2-3b)}{y^{\frac12}\left(y^3+(4-b)y^2+(6-3b)y+(4-3b)\right)^{\frac32}}\,\mathrm{d}y \\
\leq&\ \int_0^\infty\left(\frac{\left(2-\frac32b\right)y^3}{y^{\frac12}(y(y+1)^2)^{\frac32}}+\frac{(6-6b)y^2}{y^{\frac12}((y+1)^3)^{\frac32}}
+\frac{\left(6-\frac{15}2b\right)y}{y^{\frac12}((y+1)^3)^{\frac32}}+\frac{2-3b}{y^{\frac12}((y+1)^3)^{\frac32}}\right)\mathrm{d}y \\
=&\ \frac12\left(2-\frac32b\right)+\frac4{35}(6-6b)+\frac{16}{105}\left(6-\frac{15}2b\right)+\frac{32}{35}(2-3b) \\
=&\ \frac{31}7-\frac{149}{28}b\leq-\frac{25}{28}<0.
\end{align*}}
So we have $\lim\limits_{r\to\infty}\frac{\mathrm{d}\ln V(r;c(R))}{{\mathrm{d}}c}<0$, which means all maximal hypersurfaces are disjoint.

To show these maximal hypersurfaces cover the whole Kruskal extension, it suffices to show that for all fixed $r>2M$, $\lim\limits_{c\to c_H}V(r;c)=\infty$
and $\lim\limits_{c\to 0}\frac{U(r;c)}{V(r;c)}=1$.
Since $V(r;c)\geq V(2M;c)$ and $V(2M;c)\to\infty$ as $c\to c_{H}$ by Proposition \ref{prop23}, we get $\lim\limits_{c\to c_H}V(r;c)=\infty$.

From (\ref{barf}), (\ref{Vcoord}), and $c=-R^{\frac32}(2M-R)^{\frac12}$, we have
\begin{align*}
\frac{U(r;c(R))}{V(r;c(R))}
&=\mathrm{e}^{\frac1{2M}(r+2M\ln|r-2M|-R-2M\ln|R-2M|-\int_R^r\bar{f}'(x)\,\mathrm{d}x)} \\
&=\mathrm{e}^{\frac1{2M}(\int_R^r(1+\frac{2M}{x-2M}-\bar{f}'(x,c(R)))\,\mathrm{d}x)} \\
&=\mathrm{e}^{\frac{R^{\frac32}(2M-R)^{\frac12}}{2M}\left(\int_R^r\frac{x}{(x-2M)\sqrt{x-R}\sqrt{x^3-(2M-R)(x^2+Rx+R^2)}}\,\mathrm{d}x\right)}.
\end{align*}
The positive function $\frac{x}{\sqrt{x^3-(2M-R)(x^2+Rx+R^2)}}$ is bounded by $\frac{4}{\sqrt{3M}}$ as long as $R\geq\frac{7M}4$ and $x>R$.
We also have
\begin{align*}
\int\frac{1}{(x-2M)\sqrt{x-R}}\mathrm{d}x=\frac1{\sqrt{2M-R}}\ln\left|\frac{\sqrt{x-R}-\sqrt{2M-R}}{\sqrt{x-R}+\sqrt{2M-R}}\right|+C.
\end{align*}
Hence
\begin{align*}
1\leq\lim_{c\to 0}\frac{U(r;c(R))}{V(r;c(R))}
&\leq\lim_{R\to 2M}\mathrm{e}^{\frac{R^{\frac32}}{2M}\left.\ln\left|\frac{\sqrt{x-R}-\sqrt{2M-R}}{\sqrt{x-R}+\sqrt{2M-R}}\right|\right|_{x=R}^{x=r}} \\
&=\lim_{R\to 2M}\left|\frac{\sqrt{r-R}-\sqrt{2M-R}}{\sqrt{r-R}+\sqrt{2M-R}}\right|^{\frac{R^{\frac32}}{2M}}=1.
\end{align*}
In conclusion, we get the foliation theorem.
\begin{thm}
If $H=0$, the foliation Conjecture \ref{conj} is true.
\end{thm}

\appendix
\appendixpage
\section{Formula of $\frac{\mathrm{d}\ln V(r;c(R))}{\mathrm{d}c}$ in $\Sigma^-_{H,c_H<c<0}$} \label{app71}
The aim of the appendices \ref{app71} and \ref{app72} is to derive the formula of $\frac{\mathrm{d}\ln V(r;c(R))}{\mathrm{d}c}$.
First of all, we discuss \TSC hypersurfaces in $\Sigma^-_{H,c_H<c<0}$.
For $r\in[0,R)$, by the chain rule, we have
\begin{align*}
\frac{\mathrm{d}\ln V(r;c(R))}{\mathrm{d}c}
=\frac{\mathrm{d}\ln V(r;c(R))}{\mathrm{d}R}\frac{\mathrm{d}R}{\mathrm{d}c}
=\frac1{4M}\left(\frac{\mathrm{d}}{\mathrm{d}R}\int_R^rf'(x;c(R))\,\mathrm{d}x\right)\frac{\mathrm{d}R}{\mathrm{d}c},
\end{align*}
where
{\small
\begin{align*}
f'(x,c(R))=\frac1{h(x)}\sqrt{\frac{l^2(x,c(R))}{l^2(x,c(R))-1}}\quad\mbox{and}\quad l(x,c(R))=\frac1{\sqrt{-h(R)}}\left(-Hx-\frac{c(R)}{x^2}\right)
\end{align*}}
from (\ref{f2negative}) and $c(R)=-HR^3-R^{\frac32}(2M-R)^{\frac12}$.
Some rearrangements give
\begin{align*}
\frac{\mathrm{d}}{\mathrm{d}R}\int_R^rf'(x;c(R))\mathrm{d}x
&=\frac{\mathrm{d}}{\mathrm{d}R}\int_r^R\frac{A(x,R)}{h(x)\sqrt{A^2(x,R)+B(x)}}\,\mathrm{d}x \\
&=\frac{\mathrm{d}}{\mathrm{d}R}\int_r^R\frac{A(x,R)}{h(x)\sqrt{(R-x)P(x,R)}}\,\mathrm{d}x,
\end{align*}
where
$A(x,R)=Hx^3-HR^3-R^{\frac32}(2M-R)^{\frac12}, \, B(x)=x^3(x-2M)$, and  $P(x,R)\neq 0$.
Since $\int_R^rf'(x;c(R))\mathrm{d}x$ is an improper integral, we have to be careful.
For $\varepsilon>0$, define
\begin{align*}
\phi_{\varepsilon}(R)=\int_r^{R-\varepsilon}\frac{A(x,R)}{h(x)\sqrt{A^2(x,R)+B(x)}}\,\mathrm{d}x
=\int_r^{R-\varepsilon}\frac{A(x,R)}{h(x)\sqrt{(R-x)P(x,R)}}\,\mathrm{d}x.
\end{align*}
By the fundamental theorem of calculus, we have
{\small
\begin{align}
&\frac{\mathrm{d}\phi_{\varepsilon}(R)}{\mathrm{d}R} \notag \\
=\ &\frac{A(R-\varepsilon,R)}{h(R-\varepsilon)\sqrt{A^2(R-\varepsilon,R)+B(R-\varepsilon)}}
+\int_r^{R-\varepsilon}\frac{1}{h(x)}
\frac{\mathrm{d}}{\mathrm{d}R}\left(\frac{A(x,R)}{\sqrt{A^2(x,R)+B(x)}}\right)\mathrm{d}x \notag \\
=\ &\frac{1}{h(R-\varepsilon)}\int_r^{R-\varepsilon}
\left[\begin{array}{l}
\frac{\mathrm{d}}{\mathrm{d}x}\left(\frac{A(x,R)}{\sqrt{A^2(x,R)+B(x)}}\right)
+\frac{h(R-\varepsilon)}{h(x)}\frac{\mathrm{d}}{\mathrm{d}R}\left(\frac{A(x,R)}{\sqrt{A^2(x,R)+B(x)}}\right)
\end{array}\right]\mathrm{d}x \notag \\
&+\frac{1}{h(R-\varepsilon)}\frac{A(r,R)}{\sqrt{A^2(r,R)+B(r)}}. \label{formulaphi}
\end{align}}
A direct computation shows that the terms with order $(R-x)^{-\frac32}$ in the integrand all have $\varepsilon$ in their coefficients.
Therefore $\frac{\mathrm{d}\phi_{\varepsilon}(R)}{\mathrm{d}R}$ converges uniformly and it implies that
{\small
\begin{align*}
&\frac{\mathrm{d}}{\mathrm{d}R}\int_R^rf'(x;c(R))\,\mathrm{d}x \notag \\
=\ &\frac{1}{h(R)}\int_{r}^{R}
\frac{H\cdot F(x,R)+G(x,R)}
{(R-x)^{\frac12}(P(x,R))^{\frac32}}\,\mathrm{d}x+\frac{1}{h(R)}\frac{A(r,R)}{\sqrt{A^2(r,R)+B(r)}},
\end{align*}}
where
\begin{align}
F(x,R)&=x^2(-3x^2(x+R-2M)+(2x-3M)(x^2+Rx+R^2)) \notag\\
&=x^2((3M-x)(x^2-R^2)+xR(R-3M-x)) \label{Ffunction}\\
&=x^2(-x^3+(3M-R)x^2+(2R^2-3MR)x-3MR^2), \label{Ffunction2}\\
G(x,R)&=x^2\sqrt{-h(R)}(x(R-3M)+R(x-3M)). \label{Gfunction}
\end{align}

The function $c(R)=-HR^3-R^{\frac32}(2M-R)^{\frac12}$ implies
\begin{align}
\frac{\mathrm{d}R}{\mathrm{d}c}=\frac{\sqrt{-h(R)}}{-3HR^{\frac32}(2M-R)+(2R-3M)} \label{diffRc}
\end{align}
and we denote
\begin{align}
J(R)=-3HR^{\frac32}(2M-R)+(2R-3M). \label{JR}
\end{align}
In conclusion, we have
\begin{align}
&\ \frac{\mathrm{d}\ln V(r;c(R))}{\mathrm{d}c} \notag \\
=&-\frac1{4MJ(R)\sqrt{-h(R)}}\left(\int_{r}^{R}\frac{H\cdot F(x,R)+G(x,R)}{(R-x)^{\frac12}(P(x,R))^{\frac32}}\,\mathrm{d}x
+\frac{A(r,R)}{\sqrt{A^2(r,R)+B(r)}}\right).
\label{VformulaII}
\end{align}
Taking $r=0$ in (\ref{VformulaII}) and $\frac{A(0,R)}{\sqrt{A^2(0,R)+B(0)}}=-1$ give
\begin{align}
\frac{\mathrm{d}\ln V(0;c(R))}{\mathrm{d}c}
=-\frac1{4MJ(R)\sqrt{-h(R)}}\left(\int_{0}^{R}\frac{H\cdot F(x,R)+G(x,R)}{(R-x)^{\frac12}(P(x,R))^{\frac32}}\,\mathrm{d}x-1\right)
\label{VformulaII0}
\end{align}
The criteria in Proposition~\ref{prop15} implies that if
$\frac{\mathrm{d}\ln V(0;c)}{\mathrm{d}c}=\lim\limits_{r\to 0}\frac{\mathrm{d}\ln V(r;c)}{\mathrm{d}c}\leq 0$,
then hypersurfaces in $\Sigma_{H,c_H<c<0}^-$ are disjoint.
Remark that $\frac{\mathrm{d}R}{\mathrm{d}c}<0$ for $\Sigma_{H,c_H<c<0}^-$, which implies $J(R)<0$.

\section{Formula of $\frac{\mathrm{d}\ln V(r;c(R))}{\mathrm{d}c}$ in $\Sigma_{H,-8M^3H>c>c_H}^+$} \label{app72}
Next, we consider \TSC hypersurfaces in $\Sigma_{H,-8M^3H>c>c_H}^+$.
For $r\in(R,\infty)$, by the chain rule, we have $\frac{\mathrm{d}\ln V(r;c(R))}{\mathrm{d}c}
=\frac{\mathrm{d}\ln V(r;c(R))}{\mathrm{d}R}\frac{\mathrm{d}R}{\mathrm{d}c}$ and formula (\ref{Vcoord}) implies
\begin{align*}
&\ 4M\frac{\mathrm{d}\ln V(r;c(R))}{\mathrm{d}R}
=\lim_{\varepsilon\to 0}\frac{\mathrm{d}}{\mathrm{d}R}\left(\int_{R+\varepsilon}^r\bar{f}'(x;R)\,\mathrm{d}x+R+2M\ln|R-2M|\right).
\end{align*}
Since $\bar{f}'$ is smooth on $r\in[R+\varepsilon,r]$, we can use the fundamental theorem of calculus to get
\begin{align*}
&\ \ 4M\frac{\mathrm{d}\ln V(r;c(R))}{\mathrm{d}R} \\
=&\ \lim_{\varepsilon\to 0}\left(-\bar{f}'(R+\varepsilon;R)
+\int_{R+\varepsilon}^r\frac{\mathrm{d}}{\mathrm{d}R}\bar{f}'(x;R)\,\mathrm{d}x+\frac1{h(R+\varepsilon)}\right) \\
=&\ \lim_{\varepsilon\to 0}\left(-f'(R+\varepsilon;R)+\int_{R+\varepsilon}^r\frac{\mathrm{d}}{\mathrm{d}R}f'(x;R)\,\mathrm{d}x\right) \\
=&\ \lim_{\varepsilon\to 0}\left[
\begin{array}{l}
-\frac{A(R+\varepsilon,R)}{h(R+\varepsilon)\sqrt{A^2(R+\varepsilon,R)+B(R+\varepsilon)}}
+\int_{R+\varepsilon}^r\frac{1}{h(x)}
\frac{\mathrm{d}}{\mathrm{d}R}\left(\frac{A(x,R)}{\sqrt{A^2(x,R)+B(x)}}\right)\mathrm{d}x
\end{array}\right]
\end{align*}

Similar to the formula (\ref{formulaphi}) and its argument, we get
\begin{align*}
&\ 4M\frac{\mathrm{d}\ln V(r;c(R))}{\mathrm{d}R} \\
=&-\frac{1}{h(R)}\int_R^r\frac{H\cdot F(x,R)+G(x,R)}{(R-x)^{\frac12}(P(x,R))^{\frac32}}\,\mathrm{d}x-\frac{1}{h(R)}\frac{A(r,R)}{\sqrt{A^2(r,R)+B(r)}},
\end{align*}
where $F(x,R),G(x,R)$ are as (\ref{Ffunction}) and (\ref{Gfunction}), respectively.
From (\ref{diffRc}), we get
\begin{align}
&\ \frac{\mathrm{d}\ln V(r;c(R))}{\mathrm{d}c} \notag \\
=&\ \frac{1}{4M J(R)\sqrt{-h(R)}}\left(\int_{R}^r\frac{H\cdot F(x,R)+G(x,R)}
{(x-R)^{\frac12}(P(x,R))^{\frac32}}\,\mathrm{d}x+\frac{A(r,R)}{\sqrt{A^2(r,R)+B(r)}}\right), \label{criterion4}
\end{align}
where $J(R)=-3HR^{\frac32}(2M-R)+(2R-3M)$.
In particular, taking $r=2M$, formula (\ref{criterion4}) becomes
\begin{align}
\frac{\mathrm{d}\ln V(2M;c(R))}{\mathrm{d}c}
=\frac{1}{4M J(R)\sqrt{-h(R)}}\left(\int_{R}^{2M}\frac{H\cdot F(x,R)+G(x,R)}
{(x-R)^{\frac12}(P(x,R))^{\frac32}}\,\mathrm{d}x-1\right). \label{criterion2M}
\end{align}
Remark that Proposition~\ref{prop17} implies that in $\Sigma_{H,-8M^3H>c>c_H}^+$,
if the limit $\lim\limits_{r\to\infty}\frac{\mathrm{d}\ln V(r,c)}{\mathrm{d}c}\leq 0$,
then hypersurfaces are disjoint.
Furthermore, we have $\frac{\mathrm{d}R}{\mathrm{d}c}>0$, which implies $J(R)>0$.

\section{Change of variables} \label{app73}
We shall use the following change of variables for better control.
Define $a=\frac{2M}{r_H}$, $b=\frac{2M}{R}$, and $z=\frac{x}{R}$.
Then $h(R)=1-b$ and
{\small
\begin{align}
HR
=\frac{2r_H-3M}{3\sqrt{r_H^3(2M-r_H)}}R
=\frac{R}{r_H}\frac1{\sqrt{\frac{2M}{r_H}-1}}\left(\frac{4-\frac{6M}{r_H}}6\right)
=\frac{a}b\frac1{\sqrt{a-1}}\left(\frac{4-3a}6\right). \label{HRformula}
\end{align}}
Furthermore, we have
{\small
\begin{align*}
H\cdot F(x,R)&=\frac{a}b\frac{R^4z^2}{\sqrt{a-1}}\left(\frac{4-3a}6\right)\left(-3z^2(z+1-b)+\left(2z-\frac32b\right)(z^2+z+1)\right), \\
G(x,R)&=R^4z^2\sqrt{b-1}\left(z\left(1-\frac32b\right)+\left(z-\frac32b\right)\right),
\end{align*}
and
\begin{align*}
&\ P(x,R) \\
=&\ R^3\left(
\begin{array}{l}
\left(\frac{a(4-3a)}{6b(a-1)^{\frac12}}\right)^2(z-1)(z^2+z+1)^2-\frac{a(4-3a)}{3b(a-1)^{\frac12}}(z^2+z+1)(b-1)^{\frac12} \\
+(z^3+z^2+z+1)-b(z^2+z+1).
\end{array}
\right)
\end{align*}}
Hence
\begin{align}
\int_{R}^r\frac{H\cdot F(x,R)+G(x,R)}{(x-R)^{\frac12}(P(x,R))^{\frac32}}\,\mathrm{d}x
=\int_{1}^{\frac{r}R}\frac{\tilde{F}(z,a,b)+\tilde{G}(z,b)}{(z-1)^{\frac12}(\tilde{P}(z,a,b))^{\frac32}}\,\mathrm{d}z, \label{changevariable}
\end{align}
where
\begin{align*}
\tilde{F}(z,a,b)&=\frac{a(4-3a)}{6b(a-1)^{\frac12}}z^2\left(-3z^2(z+1-b)+\left(2z-\frac32b\right)(z^2+z+1)\right), \\
\tilde{G}(z,b)&=z^2(b-1)^{\frac12}\left(\left(2-\frac32b\right)z-\frac32b\right),
\end{align*}
and
\begin{align*}
&\ \tilde{P}(z,a,b)\\
=&\left(\frac{a(4-3a)}{6b(a-1)^{\frac12}}\right)^2(z-1)(z^2+z+1)^2-\frac{a(4-3a)}{3b(a-1)^{\frac12}}(z^2+z+1)(b-1)^{\frac12} \\
&+(z^3+z^2+z+1)-b(z^2+z+1).
\end{align*}
We also have
\begin{align}
\frac{A(r,R)}{\sqrt{A^2(r,R)+B(r)}}
&=\left.\frac{A(x,R)}{\sqrt{A^2(x,R)+B(x)}}\right|_{x=r} \notag \\
&=\left.\frac{\frac{a(4-3a)}{6b(a-1)^{\frac12}}(z^3-1)-(b-1)^{\frac12}}
{\sqrt{\left(\frac{a(4-3a)}{6b(a-1)^{\frac12}}(z^3-1)-(b-1)^{\frac12}\right)^2+z^3(z-b)}}\right|_{z=\frac{r}R}. \label{bdyterm}
\end{align}
These change of variables are helpful to get better estimates on the criteria (\ref{VformulaII0}) and (\ref{criterion4}).

\subsection*{Acknowledgment}
The authors would like to thank Quo-Shin Chi, Mao-Pei Tsui, and Mu-Tao Wang for their interests and discussions.
The first author also likes to express his gratitude to Robert Bartnik, Pengzi Miao and Todd Oliynyk for helpful suggestions
and hospitality when he visited Monash University.
The first author is supported by the NSC research grant 101-2917-I-564-005 and
the second author is partially supported by the NSC research grant 99-2115-M-002-008 in Taiwan.
We are also grateful to Zhuo-Bin Liang for useful comments.

\end{document}